\documentclass{amsart}

\usepackage{amsmath, amsthm, amssymb}
\usepackage[english]{babel}
\usepackage{mathrsfs}
\usepackage[all]{xy}
\usepackage{paralist}
\usepackage{graphicx}

\usepackage{hyperref}
\usepackage[dvipsnames]{xcolor}
\hypersetup{
	colorlinks=true,
	citecolor=blue!60!black,
	linkcolor=red!60!black,
	urlcolor=green!40!black,
	filecolor=yellow!50!black,
	breaklinks=true,
	pdfpagemode=UseNone,
	bookmarksopen=false,
}

\theoremstyle{plain}

\makeatletter
\newtheorem*{rep@theorem}{\rep@title}
\newcommand{\newreptheorem}[2]{%
\newenvironment{rep#1}[1]{%
 \def\rep@title{#2~\ref{##1}}%
 \begin{rep@theorem}}%
 {\end{rep@theorem}}}
 
\makeatother

\newtheorem{theorem}{Theorem}[section]
\newtheorem{definition}[theorem]{Definition}
\newtheorem{lemma}[theorem]{Lemma}
\newtheorem{proposition}[theorem]{Proposition}

\newreptheorem{theorem}{Theorem}
\newreptheorem{proposition}{Proposition}

\theoremstyle{remark}
\newtheorem{remark}[theorem]{Remark}

\numberwithin{equation}{section}

\def\XXint#1#2#3{{\setbox0=\hbox{$#1{#2#3}{\int}$ }
\vcenter{\hbox{$#2#3$ }}\kern-.6\wd0}}

\newcommand{\bfc}{{\bf c}}

\newcommand{\bfw}{{\bf w}}

\newcommand{\bfz}{{\bf z}}

\newcommand{\bbB}{\mathbb B}

\newcommand{\bbI}{\mathbb I}

\newcommand{\bbN}{\mathbb N}

\newcommand{\bbS}{\mathbb S}

\newcommand{\calE}{\mathcal E}

\newcommand{\calI}{\mathcal I}
\newcommand{\calJ}{\mathcal J}

\newcommand{\calS}{\mathcal S}

\newcommand{\bp}{\begin{pmatrix}}
\newcommand{\ep}{\end{pmatrix}}
\newcommand{\p}{\partial}

\newcommand{\N}{\mathbb{N}}
\newcommand{\R}{\mathbb{R}}

\newcommand{\T}{\mathbb{T}}

\newcommand{\ve}{\varepsilon}

\newcommand{\dx}{\textnormal{d}x}

\newcommand{\dv}{\textnormal{d}v}

\newcommand{\dt}{\textnormal{d}t}
\newcommand{\ds}{\textnormal{d}s}
\newcommand{\di}{\textnormal{d}I}

\newcommand{\To}{\longrightarrow}

\newcommand{\weakto}{\rightharpoonup}

\newcommand{\lt}{\left}
\newcommand{\rt}{\right}
\newcommand{\intr}{\int_{\R^n}}
\newcommand{\pa}{\partial}
\newcommand{\ov}{\overline}

\date{June 2024}

\begin{document}

\title[Solutions to BGK Model for Barotropic Gas Dynamics]{Global Mild Solutions to a BGK Model for Barotropic Gas Dynamics}

\author{Dowan Koo}
\address{Department of Mathematics, Yonsei University, Seoul 03722, Republic of Korea}
\email{dowan.koo@yonsei.ac.kr}

\author{Sihyun Song}
\address{Department of Mathematics, Yonsei University, Seoul 03722, Republic of Korea}
\email{ssong@yonsei.ac.kr}

\date{\today}
\keywords{Global mild solutions, BGK-type model, barotropic Euler equations, kinetic entropy inequality, hydrodynamic limit, velocity averaging}

\begin{abstract}
We establish global existence of mild solutions to the BGK model proposed by Bouchut [J. Stat. Phys., {\bf 95}, (1999), 113–170] under the minimal assumption of finite kinetic entropy initial data. Moreover we rigorously derive a kinetic entropy inequality, which combined with the theory developed by Berthelin and Vasseur [SIAM J. Math. Anal., {\bf 36}, (2005), 1807--1835] leads to the hydrodynamic limit to the barotropic Euler equations. The main tools employed in the analysis are stability estimates for the Maxwellian and a velocity averaging lemma.
\end{abstract}

\maketitle
\setcounter{tocdepth}{1}
\tableofcontents

\section{Introduction}
\subsection{System and Notions}
In this paper we study a BGK-type model, introduced and referred to in \cite[Section 3.1.2]{bouchut1999} as the first model for barotropic gas dynamics:
\begin{equation}\label{BGK}
\begin{cases}
    \displaystyle \p_t f + v \cdot \nabla_x f = \frac{1}{\tau} (M[f] - f),\\
    f(0,\cdot,\cdot)=f_0(\cdot,\cdot).
    \end{cases}
    \end{equation}
Here $f(t,x,v)$ is the distribution function expressing the number of gas particles in the infinitesimal cube of volume $\dx\dv$ located at $(x,v)$ of the phase space $\R^n\times \R^n$ (at time $t\ge 0$). The parameter $\tau>0$ is the relaxation time which we interpret here as the microscopic scale.

For each $\gamma\in\left(1,\frac{n+2}{n}\right]$ and $\kappa\ge 0$, 
the Maxwellian $M$ is first defined for a vector $(\rho,u)\in \R_+\times \R^n$ by
\begin{align}
\label{def : M}
  M[\rho,u](v) := \begin{cases}
        \displaystyle c_2\left(c_1\rho^{\gamma-1} - |v-u|^2\right)_+^{d/2} &  \displaystyle \gamma \in \left(1,\frac{n+2}{n}\right), \\[9pt]
        \displaystyle c_2 \mathbf{1}_{c_1\rho^{2/n} \ge|v-u|^2} &  \displaystyle\gamma =\frac{n+2}{n}.
    \end{cases}  
\end{align}
Here,
\begin{align*}
    d := \frac{2}{\gamma-1}-n
\end{align*}
denotes the degree of freedom, and the constants are given by
\begin{align*}
    c_1 = \frac{2\gamma\kappa}{\gamma-1},\quad  c_2 = \left(\frac{2\gamma\kappa}{\gamma-1}\right)^{-\frac{1}{\gamma-1}} \frac{\Gamma\left(\frac{\gamma}{\gamma-1}\right)}{\pi^{n/2}\Gamma\left(\frac{d}{2}+1\right)}.
\end{align*}
Then, for $f\ge0$ with $f\in L^1(\R^n;(1+|v|)\dv)$, the Maxwellian corresponding to $f$ is defined as
\begin{align*}
    M[f] := M[\rho_f,u_f],
\end{align*}
where $\rho_f$ and $u_f$ are the macroscopic density and bulk velocity associated to $f$, respectively:
\begin{align*}
    \rho_{f} := \int_{\R^n} f(v) \dv, \qquad 
    u_{f} := \begin{cases} 
    \displaystyle\frac{1}{\rho_f}\int_{\R^n} v f(v) \dv &\text{if } \rho_f \ne 0, \\[9pt]
    0 &\text{if }\rho_f = 0.
    \end{cases}
\end{align*}
We recall some basic properties of the model \eqref{BGK}. It has an associated kinetic entropy that is written
\begin{equation}\label{eq:kin:ent}
    H(f,v) := \begin{cases}
      \displaystyle\frac{|v|^2}{2}f + \frac{1}{2c_2^{2/d}} \frac{f^{1+2/d}}{1+2/d} & \displaystyle\gamma \in \left(1,\frac{n+2}{n}\right),\\[9pt]
    \displaystyle\frac{|v|^2}{2}f + \infty \cdot \mathbf{1}_{f>c_2} & \displaystyle\gamma =\frac{n+2}{n}.\\
    \end{cases}
\end{equation}
The Maxwellian is constructed to satisfy the following properties, we refer to \cite{bouchut1999,berthelinbouchut2000,choihwang2024} for explicit proofs and calculations. For any $f\ge 0$ with $f\in L^1(\R^n;(1+|v|^2)\dv)$, the Maxwellian verifies
\begin{equation*}
\begin{cases}
    \displaystyle\int_{\R^n} (1,v)M[f] \dv = (\rho_f, \rho_f u_f),\\[9pt]
    \displaystyle\int_{\R^n} v \otimes v M[f]  \, \dv = \rho_f u_f \otimes u_f + \kappa \rho_f^\gamma \bbI.
\end{cases}
\end{equation*}
Most importantly for any $f\ge 0$ such that $f+ H(f,v)\in L^1(\R^n)$ the following minimization principle
\begin{equation}\label{eq : min principle}
\int_{\R^n} H(f,v)\dv \ge \int_{\R^n} H(M[f],v)\dv
\end{equation}
and compatibility condition hold:
\begin{equation}
\label{eq: comp cond}
    \int_{\R^n} H(M[f],v)\dv = \frac{1}{2}\rho_f |u_f|^2 + \frac{\kappa}{\gamma-1} \rho_f^{\gamma}.
\end{equation}
\begin{remark}
    For the endpoint case $\gamma = \frac{n+2}{n}$ the infinity term in \eqref{eq:kin:ent} is necessary for the minimization principle \eqref{eq : min principle} to hold for all $f$. Indeed, we provide in Appendix \ref{rem:countex} a counterexample to the minimization principle if it is absent. To bypass this issue the initial data $f_0$ should be assumed to satisfy $f_0 \le c_2$, then this bound propagates with time.
\end{remark}
\begin{remark}
\label{rem : max const factor}
The Maxwellian for the end-point case $\gamma = \frac{n+2}{n}$ was written in \cite{berthelinvasseur2005, choihwang2024} as
\begin{align}
\label{eq : calM}
    \mathcal{M}[f] := \mathbf{1}_{\left\{ \kappa n\rho_f^{2/n}/|\bbS_n| \ge |u_f - v|^2 \right\}}.
\end{align}
In this way
\begin{equation*}
    \int_{\R^n} H(\mathcal{M}[f],v) \dv = \frac{1}{2}\rho_f |u_f|^2 + \frac{\kappa n C_n}{2}\rho_f^{\frac{n+2}{n}}
\end{equation*}
for a constant $C_n\ne 1$ so one needs to modify \eqref{eq: comp cond} for $\gamma=\frac{n+2}{n}$. In this work, we tweaked the constants arising in the Maxwellian (compare \eqref{eq : calM} with our definition \eqref{def : M}) so that \eqref{eq: comp cond} holds consistently for all $\gamma\in\left(1,\frac{n+2}{n}\right]$.
\end{remark}

Finally we mention that for a solution $f\ge 0$ to \eqref{BGK} a formal calculation leads to the \emph{kinetic entropy inequality}
\begin{equation}
    \label{eq : kin entrop ineq}
    \begin{split}
          &\iint_{\R^{2n}} H(f,v)\dx \dv + \frac{1}{\tau}\int_0^t \iint_{\R^{2n}} H(f,v)-H(M^{}[f],v) \dx \dv \ds\\
          &\qquad \le \iint_{\R^{2n}}H(f_0,v)\dx\dv.
        \end{split}
\end{equation}
For initial data $f_0$ with finite kinetic entropy this implies
    \begin{align}
    \label{ineq : dissipation est}
        \int_0^t \iint_{\R^{2n}} H(f,v)-H(M[f],v) \dx \dv \ds \le \tau \iint_{\R^{2n}} H(f_0,v) \dx \dv = C\tau.
    \end{align}
This estimate is crucially used in deriving the hydrodynamic limit of \eqref{BGK} as we explain in Section \ref{subs: hyd}.

\subsection{History and Review}
The BGK model was proposed independently in \cite{bhatnagargrosskrook1954} and \cite{walender1954} as a relaxation time approximation of the Boltzmann equation \cite{boltzmann1995}. It describes relaxation towards equilibrium with a simpler collision operator yet obeys certain conservation laws that the Boltzmann equation is founded upon. Recently, interest has been given to its applications to gas mixtures \cite{andriesaokiperthame2002, baeklingenbergpirneryun2021, bobylevbisigroppispigapotapenko2018, brullpavanschneider2012, klingenbergpirnerpuppo2017}, relativistic particles \cite{andersonwitting1974, pennisiruggeri2018}, reactive gas molecules \cite{kimleeyun2021}, and so on.

For our model of interest \eqref{BGK} we mention the following works. Its construction can be found in Bouchut \cite{bouchut1999} for the specific study of barotropic gas dynamics. The hydrodynamic limit of \eqref{BGK} to the barotropic Euler equations was discussed in \cite{berthelinbouchut2002,berthelinvasseur2005}. Existence of solutions for the monodimensional case was studied in \cite{berthelinbouchut2000}. More recently \cite{choihwang2024} proved existence of weak solutions to \eqref{BGK} in dimensions $n\ge 2$ for the range $\gamma \in \left(1,\frac{n+4}{n+2}\right]\cup \left\{\frac{n+2}{n}\right\}$. Meanwhile, construction of unique classical solutions to \eqref{BGK} near a global Maxwellian was obtained in \cite{hwang2024}, in a more restricted range of $\gamma$. Beyond these developments the existence of solutions in the full range $\gamma\in\left(1,\frac{n+2}{n}\right]$ is yet to be established.

\subsection{Hydrodynamic Limit of \eqref{BGK} to the Barotropic Euler Equations} \label{subs: hyd}
Berthelin and Vasseur \cite{berthelinvasseur2005} studied the hydrodynamic limit of \eqref{BGK} to the compressible Euler equations
\begin{equation}
\label{eq:BE}
    \begin{cases}
        \p_t \rho + \nabla\cdot (\rho u) = 0,\\
        \p_t (\rho u) + \nabla \cdot (\rho u \otimes u + \kappa \rho^\gamma \bbI) = 0,
    \end{cases}
\end{equation}
where $\gamma$ is the heat capacity ratio, and $\kappa$ is the proportionality constant for the adiabatic relations of an ideal gas. They deduced the following result, which links \eqref{BGK} to \eqref{eq:BE}.
\begin{theorem}
\label{thm : hyd}
\cite[Theorem 1.1]{berthelinvasseur2005}. Let $\gamma \in (1,\frac{n+2}{n}]$ and $(\rho_0,\rho_0 u_0) \in L^1(\R^n)$ be the given initial data for the solution $(\rho, \rho u)\in C^1([0,T)\times\R^n) \cap L^1([0,T)\times \R^n)$ to \eqref{eq:BE} satisfying $\rho>0$; and $\rho,u,\nabla_x u, \nabla_x \rho$ are bounded; $\rho |u|^2,\rho^\gamma$ are integrable with respect to $(t,x)$. Consider a family of initial values $\{f^\tau_0\}_{\tau>0}$ with $f^\tau_0 + H(f^\tau_0,v)\in L^1(\R^{2n})$. Let $\{f^\tau\}_{\tau>0}$ denote corresponding solutions to \eqref{BGK} satisfying the kinetic entropy inequality \eqref{eq : kin entrop ineq}, with initial datum $f_0^\tau$. If the initial datum are \textit{well-prepared}
    \begin{align*}
        \int_{\R^n} (f_0^\tau, vf_0^\tau, H(f_0^\tau,v)) \dv \overset{\tau\to 0}{\To} \Big(\rho_0, \rho_0 u_0, \frac{1}{2}\rho_0 |u_0|^2 + \frac{\kappa}{\gamma-1}\rho_0^\gamma \Big) \quad \text{in } L^1(\R^n),
    \end{align*}
    then
    \begin{align*}
        \begin{cases}
            \rho_{f^\tau} \to \rho &\text{strongly in }C^0(0,T;L^p_{\rm loc}(\R^n)) \quad p\in \left[1,\gamma\right),\\[6pt]
            \rho_{f^\tau} u_{f^\tau} \to \rho u &\text{strongly in }C^0(0,T;L^q_{\rm loc}(\R^n)) \quad \displaystyle q\in \left[1,\frac{2\gamma}{\gamma+1}\right).\\
        \end{cases}
    \end{align*}
\end{theorem}
Let us briefly explain how the kinetic entropy inequality \eqref{eq : kin entrop ineq} is employed in \cite{berthelinvasseur2005} to yield the hydrodynamic limit of \eqref{BGK}. From the study of the BGK structure in \cite[Section 4.1.1]{berthelinvasseur2005} we note that \cite[Theorem 1.2]{berthelinvasseur2005} immediately applies to \eqref{BGK} provided that it is possible to obtain the decay estimate
\begin{equation}\label{eq:J}
        \int_0^t\iint_{\R^{2n}} |v|^2 |f- M[f]|\dv \dx \ds = O(\sqrt{\tau}) \quad \text{as} \quad \tau \to 0.
\end{equation}
For the end-point case $\gamma = \frac{n+2}{n}$ a novel idea was suggested in \cite[Proposition 4.1]{berthelinvasseur2005} to obtain the estimate
\begin{equation}
\label{eq : imp term}
\begin{split}
    &\int_{\R^n} |v|^2 |f - M[f]| \dv \lesssim (\rho_f)^{\frac{n+2}{2n}}\sqrt{D_f} + D_f,\\
    &D_f := \int_{\R^n} H(f,v) - H(M[f],v) \, \dv.
    \end{split}
\end{equation}
This estimate combined with the consequence \eqref{ineq : dissipation est} of the kinetic entropy inequality yielded \eqref{eq:J}.

For the intermediate range $\gamma \in \left(1,\frac{n+2}{n}\right)$ an analogue of \eqref{eq : imp term} (see Lemma \ref{lem b3} for the precise statement) was obtained in \cite[Section 4.1.2]{berthelinvasseur2005} for another model, namely the \emph{second kinetic model for barotropic gas dynamics} \cite[Section 3.1.3]{bouchut1999}:
\begin{align}
\label{eq : ext mod}
    \displaystyle\p_t g + v \cdot \nabla_x g = \frac{1}{\tau}(\overline{M}[g] - g).
\end{align}
Here $g = g(t,x,v,I):\R_+\times \R^n\times\R^n\times \R_+ \to [0,c_3]$ and the second Maxwellian $\overline{M}$ is defined for $(\rho,u)\in \R_+\times \R^n$ by
\begin{align}
\label{def : M bar}
    \overline{M}[\rho,u](v,I) := c_3 \mathbf{1}_{|v-u|^2 + I^2 < c_1 \rho^{\gamma-1}} \quad (v,I) \in \R^n \times [0,\infty)
\end{align}
where
\begin{align*}
    c_0=\frac{2\pi^{d/2}}{\Gamma(d/2)} \quad \text{and} \quad c_3 = \left(\frac{2\pi\gamma\kappa}{\gamma-1}\right)^{-\frac{1}{\gamma-1}} \Gamma\lt(\frac{\gamma}{\gamma-1}\rt).
\end{align*}
Then $\overline{M}[g] := \overline{M}[\rho_{g}, u_{g}]$ where
\begin{align*}
(\rho_{g},\rho_{g}u_{g}) := \displaystyle \int_{\R^{n}} \int_0^\infty (1,v)g(t,x,v,I)c_0 I^{d-1} \textnormal{d}I \dv.
\end{align*}
The first model \eqref{BGK} is recovered by integrating the second model \eqref{eq : ext mod} against $c_0 I^{d-1} \textnormal{d}I$.

The authors could not locate in \cite{berthelinvasseur2005} an analogue of \eqref{eq : imp term} for the first model in the range $\gamma\in \left(1,\frac{n+2}{n}\right)$. To complement the proof of \cite[Theorem 1.1]{berthelinvasseur2005}, we provide the following lemma.

\begin{lemma}
\label{lem:dis}
For $\gamma \in \lt(1,\frac{n+2}{n}\rt)$ and $0 \le f\in L^1(\R^n;(1+|v|^2)\dv)$, there is a constant $C=C_{n,\gamma}>0$ such that
\[
\intr |v|^2 |f-M[f]| \dv \le C \lt\{(\rho_{f})^{\frac{\gamma}{2}}\sqrt{D_f} + D_f \rt\},
\]
where
\[
D_f := \intr \lt(H(f,v) - H(M[f],v)\rt)\dv.
\]
\end{lemma}
A short proof is presented in Appendix \ref{sec : proof of lem dis}. Applying Lemma \ref{lem:dis} and \eqref{ineq : dissipation est} to the first model \eqref{BGK} we can verify that \eqref{eq:J} holds for all $\gamma \in \left(1,\frac{n+2}{n}\right]$ and Theorem \ref{thm : hyd} can be deduced.

To complete the hydrodynamic limit theory for barotropic gas dynamics initiated in \cite{berthelinvasseur2005}, we aim to answer in the positive the following question:

\vspace{.1cm}
\begin{center}
\emph{For any given $f_0$ with finite kinetic entropy, can we construct global-in-time weak solutions to \eqref{BGK} satisfying the kinetic entropy inequality?}
\end{center}

\subsection{Main Results}
The weakest possible assumption we can impose on $f_0$ is that it has finite kinetic entropy, namely
\begin{equation}
    \label{eq : init data f0 cond}
    f_0 \ge 0, \quad \iint_{\R^{2n}} (1+|v|^2) f_0  \dx\dv < \infty, \quad \text{and} \quad 
      \begin{cases}
      \|f_0\|_{L^{1+2/d}} <\infty &\displaystyle\gamma \in \lt(1, \frac{n+2}{n}\rt), \\[9pt]
      0\le f_0\le c_2 &\displaystyle \gamma = \frac{n+2}{n}.
      \end{cases}
\end{equation}
Solutions that obey the kinetic entropy inequality \eqref{eq : kin entrop ineq} should satisfy 
\begin{align}
\label{eq : f prop}
        \begin{cases}
            f \in L^\infty([0,\infty);L^1_2 \cap L^{1+2/d} (\R^{2n})) &\displaystyle \gamma \in \left(1,\frac{n+2}{n}\right),\\[9pt]
            f\in L^\infty([0,\infty);L^1_2(\R^{2n})) \quad\text{and} \quad 0 \le f \le c_2 &\displaystyle \gamma = \frac{n+2}{n}
        \end{cases}
\end{align}
where $L^1_2(\R^{2n}) := L^1(\R^{2n}, (1+|v|^2)\dx\dv)$.

Taking the above into consideration we define the following notion of a solution to \eqref{BGK}:
\begin{definition}
\label{def:mild}
Given $f_0\ge 0$, we say $f$ satisfying \eqref{eq : f prop} is a \textbf{mild solution} to \eqref{BGK} if it holds that
    \begin{align}
    \label{defeq : mild}
    f(t,x,v) = e^{-t/\tau} f_0(x-vt,v) + \frac{1}{\tau}\int_0^t e^{-(t-s)/\tau} M[f](s,x-v(t-s),v) \ds. 
    \end{align}    
    Moreover, if a mild solution $f$ verifies the kinetic entropy inequality \eqref{eq : kin entrop ineq} for all $t\in [0,\infty)$ we say $f$ is a \textbf{mild entropy solution}.
\end{definition}
We now state our main result concisely as follows.

\begin{theorem}
\label{thm : main theorem 1}
Let $n\ge 1$ and assume the initial data $f_0$ verifies \eqref{eq : init data f0 cond}. For any $\gamma \in \left(1,\frac{n+2}{n}\right]$ and $\kappa,\tau>0$ the Cauchy problem \eqref{BGK} admits a mild entropy solution.
\end{theorem}

\begin{remark}
    We emphasize that no additional assumptions on $f_0$ are imposed, namely we do not require compact support, rapid decay, or higher moments in $v$. In particular, we do not impose any assumption on the spatial moments such as
    \begin{align*}
        \iint_{\R^{2n}} |x|^2 f_0 \, \dx \dv < \infty.
    \end{align*}
\end{remark}

\begin{remark} \label{rem: weak}
    Any mild solution in the sense of Definition \ref{def:mild} is a solution to \eqref{BGK} in the distributional sense. Namely, for any $T>0$ and $\phi\in C^1_c([0,T)\times \R^{2n})$ with $\phi_0:= \phi(0,\cdot,\cdot)$
    \begin{align*}
        &-\iint_{\R^{2n}} f_0 \phi_0 \dx\dv - \int_0^T \iint_{\R^{2n}} f(\p_t\phi + v\cdot \nabla_x\phi)\dx\dv\dt\\
        &\quad = \frac{1}{\tau}\int_0^T \iint_{\R^{2n}}(M[f]-f)\phi\, \dx \dv \dt.
    \end{align*}
    This can be confirmed by a direct computation which we provide in Appendix \ref{app: uniqueness}.
\end{remark}



\subsection{Strategy of Proof}
\subsubsection{Stability Estimate for the Maxwellian}
As done in a classical approach in \cite{perthamepulvirenti1993} we derive a stability estimate for the Maxwellian (Lemma \ref{thm : theta}) in the $|v|^2$--weighted space $L^1_2(\R^{2n})$. This estimate is employed to establish the existence of solutions to an approximate system corresponding to \eqref{BGK}. In \cite{choihwang2024}, for the end-point case $\gamma = \frac{n+2}{n}$ a geometric argument was used to obtain the stability estimate. For the intermediate range $1<\gamma<\frac{n+2}{n}$ the mean value theorem was employed, but the range $\gamma \in \left(\frac{n+4}{n+2},\frac{n+2}{n}\right)$ (equivalently $\frac{d}{2} \in (0,1)$) could not be handled due to the singularities that arise at the boundary of the support of the Maxwellian.

Intriguingly, our ideas come from the second kinetic model for barotropic gas dynamics \eqref{eq : ext mod}. We do not discuss explicitly the second model but only employ the second Maxwellian $\overline{M}$ 
 \eqref{def : M bar}. For $\gamma\in\left(1,\frac{n+2}{n}\right)$ the relation between $M$ and $\overline{M}$ is given by
\begin{equation}\label{eq:M:key}
M[\rho, u](v) = c_2\left(c_1 \rho^{\gamma-1} - |v-u|^2\right)_+^{d/2} = \int_{0}^{\infty} \ov{M}[\rho,u](v,I) c_0 I^{d-1} \di.
\end{equation}
Thus instead of tackling $M$ itself we exploit \eqref{eq:M:key} to bypass the difficulties that arose in \cite{choihwang2024}. Precisely, we prove in this paper the following lemma.

\begin{lemma}
\label{thm : theta}
Let $n\ge1$. For any $\gamma \in \lt(1,\frac{n+2}{n}\rt]$, $\theta\in [0,1]$, and non-negative $f,g\in L^1(\R^n;(1+|v|)\dv)$, the following stability estimate
\begin{equation}\label{eq:M:sta}
\intr \lt|M[\rho_f,u_f]-M[\rho_g, u_g]\rt| \dv \le  |\rho_f -\rho_g | + C^\theta \min\{\rho_f,\rho_g\}^{1-\frac{\theta(\gamma-1)}{2}}|u_f - u_g|^\theta
\end{equation}
holds for some constant $C= C(\kappa)>0$. As a result, if we further assume the bounds $\rho_f, \rho_g, |u_f|, |u_g| \le C_0$ with $C_0>0$, then
\begin{equation}\label{eq:key:sta}
\intr (1+|v|^2)\lt|M[\rho_f, u_f]-M[\rho_g,u_g]\rt| \dv \le C(C_0,\gamma,\kappa)\lt( |\rho_f -\rho_g | +|u_f - u_g| \rt).
\end{equation}
\end{lemma}

\subsubsection{Analysis of an Approximate System} \label{sssec : app}
    
    \medskip
   
   The estimate \eqref{eq:key:sta} requires that the macroscopic density and bulk velocity be bounded. Thus to apply \eqref{eq:key:sta} it is necessary to introduce a modified Maxwellian $M^{(\ve)}$, which automatically ensures \eqref{eq : diff of reg quantities}. We establish the existence of approximate solutions $f^\ve$ corresponding to $M^{(\ve)}$ via Picard iteration and then obtain uniform estimates for $f^\ve$.
   
   While similar in nature to that of \cite{choihwang2024}, our choice of modified Maxwellian $M^{(\ve)}$ is distinct in that it additionally guarantees the pointwise relations \eqref{eq:keydif}, which play an indispensable role throughout this work. As an instance, via Lemma \ref{lem:tight}, the relations in \eqref{eq:keydif} allow us to obtain tightness for $\rho_{f^\ve}$ without any assumptions on the spatial moments of the initial data.
   
In passing to the limit, compactness of the macroscopic variables 
is obtained by a version of the velocity averaging lemma (Proposition \ref{Prop : Vel Avg Lem}), then the strong convergence of $M^{(\ve)}[f^\ve]$ follows by the Vitali convergence theorem.  As our solutions are constructed via mild form, the strong convergence of the Maxwellians results in strong convergence of $f^\ve$ itself. We deduce that the limit $f$ is a mild solution to \eqref{BGK}.
   
   In order to derive the kinetic entropy inequality we first prove it for the approximate solutions $f^\ve$.
   Passing to the limit in the entropy inequality is a very delicate task. The pointwise relations \eqref{eq:keydif}, the minimization principle \eqref{eq : min principle}, the pointwise convergence of $f^\ve$ (up to a subsequence via the strong convergence), and iterated uses of Fatou's lemma are all crucially employed when passing to the limit.


\subsection{Organization of the Paper}

In Section \ref{sec:sta} we prove the stability estimates given in Lemma \ref{thm : theta}. We consider the approximate system regarding the modified Maxwellian in Section \ref{sec:app}. The mild solution to \eqref{BGK} is obtained in Section \ref{sec:pss} by passing to the limit of the approximate solutions with help of a velocity averaging lemma. We rigorously establish the kinetic entropy inequality in Section \ref{sec:kei}, thus finishing the proof of Theorem \ref{thm : main theorem 1}.

Appendix \ref{rem:countex} discusses the minimization principle for the case $\gamma=\frac{n+2}{n}$. We present a proof of Lemma \ref{lem:dis} in Appendix \ref{sec : proof of lem dis}. Finally, in Appendix \ref{app: uniqueness}, we confirm that mild solutions to \eqref{BGK} satisfy the weak formulation of the equation. 

\section{Stability Estimate of the Maxwellian}\label{sec:sta}
In this section we prove Lemma \ref{thm : theta} regarding the stability of the Maxwellian. To simplify the presentation of the proofs we define for $r\in\R_+$ and $\bfc \in \R^n$
\begin{align*}
    &B[r,\bfc](v) := \mathbf{1}_{|v-\bfc|^2 \le r^2}(v)
\end{align*}
where $v \in \R^n$. For the ball in $\R^n$ with radius $r$ and center $\bfc$ we will write $B_r(\bfc)$. As a preliminary step we provide a stability estimate for $B$ with respect to the center $\bfc$.
\begin{lemma}
\label{lem : 2.1}
    Let $n\ge1$. For any $r >0$ and $\bfc_1,\bfc_2\in \R^n$ it holds that
\begin{align*}
  \int_{\R^n} |B[r,\bfc_1](v) - B[r,\bfc_2](v)| \dv \le  2|\bbB_{n-1}| r^{n-1}|\bfc_1-\bfc_2|,
\end{align*}
where $|\bbB_{n-1}|$ is the Lebesgue measure of the $(n-1)$-dimensional unit ball. For convention we take $|\bbB_0|=1$.
\end{lemma}

\begin{proof}
We treat the case $n \ge 2$ first. It is clear that
\[
\intr \lt| B[r,\bfc_1](v)- B[r,\bfc_2](v) \rt| \dv = |B_{r}(\bfc_1) \bigtriangleup B_{r}(\bfc_2)|.
\]
Denote $B_i:=B_r(\bfc_i)$, $i=1,2$ for brevity. For each $\bfz \in \pa B_1$ satisfying $(\bfz-\bfc_1) \cdot (\bfc_2-\bfc_1) \le 0$ set $\bfz':= \bfz + \bfc_2-\bfc_1$. Note that $\bfz' \in \pa B_2$ with $(\bfz' - \bfc_2) \cdot (\bfc_2-\bfc_1) = (\bfz-\bfc_1) \cdot (\bfc_2-\bfc_1) \le 0$. Indeed
\[
|\bfz' - \bfc_2| = |\bfz - \bfc_1| = r, \,\,\,\text{and}\,\,\, |\bfz - \bfz'| = |\bfc_2-\bfc_1|
\]
so that $\bfc_1\bfc_2\bfz'\bfz$ forms a parallelogram (see \textbf{Figure 1}). We claim that
\[
B_1 \setminus B_2 \subset \{\bfw \in \R^n : \bfw \in L(\bfz,\bfz')\,\, \text{for some} \,\,\bfz \in \pa B_1, (\bfz-\bfc_1) \cdot (\bfc_2-\bfc_1) \le 0\} =:\calS_1
\]
where $L(\bfz,\bfz')$ is the set of points which consists the line segment with endpoints $\bfz$ and $\bfz'$. If $\bfw \in B_1 \setminus B_2$, we note that there are two $\theta \in \R$ satisfying
\begin{equation}\label{eq:wbm}
\bfw - \theta(\bfc_2-\bfc_1) \in \pa B_1
\end{equation}
with opposite sign. We denote by $\bfz_\bfw:= \bfw - \theta_+(\bfc_2-\bfc_1)$ the unique point satisfying \eqref{eq:wbm} with $\theta_+>0$. Observe that the following linear function 
\[
w(\theta):= \bfz_\bfw + \theta(\bfc_2-\bfc_1),\,\,\, \theta\in [0,\infty)
\]
satisfies
\[
w(0)=\bfz_\bfw, \,\,\, w(1)=\bfz_\bfw',\,\,\, w(\theta_+)=\bfw.
\]
Upon recalling 
\[
|\bfw - \bfc_1|\le r = |\bfz_\bfw - \bfc_1| = |\bfz_\bfw' - \bfc_2|  < |\bfw - \bfc_2|, 
\]
we deduce that $\bfw$ lies in the line segment with endpoints $\bfz_\bfw$ and $\bfz_\bfw'$. This means $\bfw \in \calS_1$, so the claim $B_1\setminus B_2 \subset \calS_1$ is proven. Hence, it follows that 
\[
|B_1\setminus B_2| \le |\calS_1| = |\mathbb{B}_{n-1}|r^{n-1}|\bfc_1-\bfc_2|.
\]
\begin{center}
\includegraphics[scale=0.6]{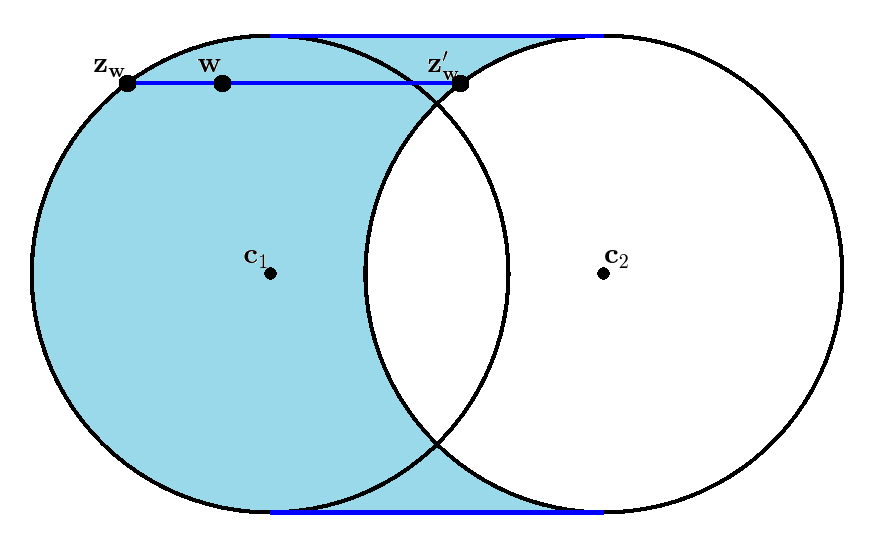}

\textbf{Figure 1 :} Illustration of $\mathbf{z}_\bfw, \bfw,$ and $\bfz_\bfw'$. The region $\calS_1$ is shaded.
\end{center}

Symmetrically we can construct a set $\calS_2$ from which we may deduce
\[
|B_1 \bigtriangleup B_2 | \le |\calS_1|+|\calS_2| = 2|\mathbb{B}_{n-1}|r^{n-1}|\bfc_1-\bfc_2|
\]
and this concludes the proof for $n\ge 2$. 

For the case $n=1$, we note that if $B[r,\bfc_1](v) \cap B[r,\bfc_2] \ne \varnothing$, then
\[
 \int_{\R} |B[r,\bfc_1](v) - B[r,\bfc_2](v)| \dv =  2|\bfc_1-\bfc_2|.
\]
On the other hand, if the two sets are disjoint, then $2r\le |\bfc_1 - \bfc_2|$. Thus:
\begin{align*}
    \int_{\R} |B[r,\bfc_1](v) - B[r,\bfc_2](v)| \dv = 4r \le 2|\bfc_1 - \bfc_2|.
\end{align*}
Altogether, the conclusion of the lemma follows.

\end{proof}

Consequently we can now obtain the stability estimates as asserted in Lemma \ref{thm : theta}.


\begin{proof}[Proof of Lemma \ref{thm : theta}]
We consider the case $\gamma \in \lt(1,\frac{n+2}{n}\rt)$: without loss of generality, we may assume $\rho_1 \ge \rho_2$.

First, we discuss the case where both densities are nonzero, in other words $\rho_2 >0$. Write
\begin{equation*}
\begin{split}
 &\intr |M[\rho_1,u_1]-M[\rho_2,u_2]| \dv \\ & \quad \le \intr |M[\rho_1,u_1]-M[\rho_2,u_1]| \dv + \intr |M[\rho_2,u_1]-M[\rho_2,u_2]| \dv \\
 &\quad =: \calI_1 + \calI_2.
\end{split}
\end{equation*}
By utilizing the assumption $\rho_1 \ge \rho_2$, we obtain
\[
\calI_1 = \intr M[\rho_1,u_1]-M[\rho_2,u_1] \dv = \rho_1 - \rho_2.
\]
To estimate $\calI_2$, we crucially exploit the following relation
\[
M[\rho_2, u_i](v) = c_2d \int_0^{\lt(c_1\rho_2^{\gamma -1}\rt)^{\frac{1}{2}}} B\lt[\lt(c_1\rho_2^{\gamma -1} -I^2\rt)^{1/2}, u_i\rt](v) I^{d-1}\di, \quad i=1,2,
\]
which follows from \eqref{eq:M:key} (note that $c_0c_3 = c_2d$).
Writing $A = c_1 \rho_2^{\gamma-1}$ and letting $\beta(x,y)$ denote the Beta function,
{\allowdisplaybreaks
\begin{align}
\label{eq : i2}
\calI_2 &= c_2d \intr \lt|\int_0^{\sqrt{A}} B\lt[\lt( A -I^2\rt)^{1/2}, u_1\rt](v) - B\lt[\lt(A -I^2\rt)^{1/2}, u_2\rt](v) I^{d-1}\di \rt| \dv \nonumber\\
&\le c_2d \int_0^{\sqrt{A}} \intr \lt| B\lt[\lt(A -I^2\rt)^{1/2}, u_1\rt](v) - B\lt[\lt(A -I^2\rt)^{1/2}, u_2\rt](v)  \rt| \dv I^{d-1}\di \nonumber\\
&\le 2|\bbB_{n-1}|c_2d \int_0^{\sqrt{A}} (A -I^2)^{\frac{n-1}2} |u_1-u_2| I^{d-1}\di \qquad  \text{(by Lemma \ref{lem : 2.1})} \nonumber\\
&=|\bbB_{n-1}|c_2d A^{\frac{n+d}{2}-\frac{1}{2}}\int_0^1 (1-s)^{\frac{n+1}{2}-1}s^{\frac{d}{2}-1}\ds |u_1-u_2| \qquad (s:=I^2/A) \nonumber\\
&= |\bbB_{n-1}|c_2 d \,\, \beta\Big(\frac{n+1}{2},\frac{d}{2}\Big)\, c_1^{\frac{1}{\gamma-1}} c_1^{-\frac{1}{2}}\rho_2^{1-\frac{\gamma-1}{2}}|u_1-u_2| \nonumber\\
&= \Lambda \, \rho_2^{1-\frac{\gamma-1}{2}}|u_1-u_2| 
\end{align} }
where $\Lambda = \Lambda(\gamma,\kappa)>0$ is given by 
\[
\Lambda = \frac{2}{\sqrt{\pi}}\frac{\frac{1}{\gamma-1}\Gamma(\frac{1}{\gamma-1})}{\Gamma(\frac{1}{\gamma-1}+\frac{1}{2})}\frac{(\gamma-1)^{\frac{1}{2}}}{(2\gamma\kappa)^{\frac{1}{2}}},
\]
calculated thanks to the relations
\[
|\bbB_{n-1}|=\frac{\pi^{\frac{n-1}{2}}}{\Gamma(\frac{n+1}{2})},\quad \beta\lt(\frac{n+1}{2},\frac{d}{2}\rt) = \frac{\Gamma(\frac{n+1}{2})\Gamma(\frac{d}{2})}{\Gamma(\frac{n+d+1}{2})}, \quad \Gamma\lt(x+1\rt)= x\Gamma(x).
\]
Using the asymptotic formula
\[
\Gamma(x+\alpha) \sim \Gamma(x) x^\alpha
\]
as $x \to \infty$, letting $\gamma \to 1^+$ we have
\[
\Lambda \to \lt(\frac{2}{\pi \kappa}\rt)^{\frac{1}{2}},
\]
which means $\Lambda$ is bounded independently of $\gamma \in \left(1,\frac{n+2}{n}\right)$, in other words $\Lambda\le C = C(\kappa)$.

Let us proceed. For any $\theta \in [0,1]$, in the case $|u_1-u_2|< C^{-1}\rho_2^{\frac{\gamma-1}{2}}$ we can estimate \eqref{eq : i2} further using $\Lambda \le C$ by
\[
\calI_2 \le C^\theta \rho_2^{1- \frac{\gamma-1}{2}\theta} |u_1-u_2|^{\theta}.
\]
In the opposite case $|u_1-u_2| \ge C^{-1}\rho_2^{\frac{\gamma-1}{2}}$, that is $C \rho_2^{-\frac{\gamma-1}{2}}|u_1-u_2|\ge 1$, it is immediate that
\[
\calI_2 = \intr |M[\rho_2,u_1]-M[\rho_2,u_2]| \dv \le 2\rho_2 \le2 C^\theta \rho_2^{1- \frac{\gamma-1}{2}\theta}|u_1-u_2|^{\theta}.
\]
This proves \eqref{eq:M:sta}, under the assumption of nonzero densities, when $\gamma \in \lt(1,\frac{n+2}{n}\rt)$. The end point case $\gamma=\frac{n+2}{n}$ can be proven similarly.

Next, we prove \eqref{eq:M:sta} for the case where a zero density is present: assume $\rho_1 \ge \rho_2 = 0$. In this case, observe that by definition $M[\rho_2,u_2](v) \equiv 0$. This yields
\begin{align*}
    \int_{\R^n} |M[\rho_1,u_1] - M[\rho_2,u_2]| \,\dv &= \int_{\R^n} M[\rho_1,u_1] \,\dv = \rho_1 = |\rho_1 - \rho_2|,
\end{align*}
and we conclude that \eqref{eq:M:sta} holds in all possible cases.

Finally, by noting that
\[
v \in \text{supp} (M[\rho,u]) \implies |v| \le |u| + c_1\rho^{\gamma-1}
\]
if we further assume $\rho_1, \rho_2, |u_1|, |u_2| \le C_0$, then
\[
v \in \text{supp} (M[\rho_1,u_1]) \cup \text{supp} (M[\rho_2,u_2])\implies |v| \le C_0 + c_1 C_0^{\gamma-1}.
\]
This observation and the inequality \eqref{eq:M:sta} applied with $\theta = 1$ leads to \eqref{eq:key:sta}.

\end{proof}

\section{Approximate Solutions}\label{sec:app}

As mentioned in Section \ref{sssec : app} the stability estimate \eqref{eq:key:sta} is not immediately applicable to obtain the continuity $f \mapsto M[f]$ in $L^1_2$. Indeed even if say $f\in L^1_2 \cap L^\infty$ it is in general not expected that the macroscopic quantities $\rho_f$ or $u_f$ be bounded. This is what motivates the modification of the Maxwellian $M \mapsto M^{(\ve)}$ with a small parameter $\ve>0$.

Thus, we would now like to consider an approximate model where a modified Maxwellian is present. We set for $0<\ve<1$
\begin{align}
\label{eq : reg bgk}
    \begin{cases}
        \p_t f^\ve + v \cdot \nabla_x f^\ve = M^{(\ve)}[f^\ve] - f^\ve, \qquad f^\ve(0,x,v) = f_0(x,v),\\[6pt]
        M^{(\ve)}[f^\ve] := M\left[\rho_{f^\ve}^{(\ve)}, u_{f^\ve}^{(\ve)}\right],\\[6pt]
        \rho_{f^\ve}^{(\ve)} := \dfrac{\rho_{f^\ve}}{1 + \ve \rho_{f^\ve}}, \qquad u_{f^\ve}^{(\ve)} := \dfrac{\rho_{f^\ve}u_{f^\ve}}{\rho_{f^\ve} + \ve (1 + |\rho_{f^\ve}u_{f^\ve}|)}.
    \end{cases}
\end{align}

Using the stability estimates for $M$ derived in the previous section, we now aim to prove the existence of solutions to \eqref{eq : reg bgk}. From hereon we present proofs for our results only for the case where $\gamma \in \big(1,\frac{n+2}{n}\big)$. For any of the sections to follow the case $\gamma=\frac{n+2}{n}$ is easier, and we did not mention it because the presentation would become untidy. For example, we would have to repeatedly change the definition of $H(f,v)$.

Also, for the sake of convenience we shall from now set $\kappa = \tau =1$ because they play no role when discussing the existence of solutions.

Note as a preliminary step that for any $f,g$ it holds that
\begin{equation}
\label{eq : diff of reg quantities}
\begin{split}
\rho_{f}^{(\ve)} \le \frac{1}{\ve},\quad |u_{f}^{(\ve)}| \le \frac{1}{\ve},\quad|\rho_{f}^{(\ve)} - \rho^{(\ve)}_g| \le |\rho_f - \rho_g|,\\
 \quad |u^{(\ve)}_f - u^{(\ve)}_g| \le \frac{2|\rho_f u_f - \rho_g u_g|}{\ve} + \frac{|\rho_g - \rho_f|}{\ve^2}.
\end{split}
\end{equation}
These relations and \eqref{eq:key:sta} imply that for any $f,g\in L^1_2(\R^{2n})$:
\begin{align}
\label{eq : max f g stab}
    \iint_{\R^{2n}}(1+|v|^2)|M^{(\ve)}[f] - M^{(\ve)}[g]| \dx \dv \le C(\ve)\iint_{\R^{2n}} (1+|v|^2)|f - g| \dx \dv.
\end{align}
We remark that \eqref{eq : diff of reg quantities} is also satisfied by the modified Maxwellian considered in \cite{choihwang2024}. The main difference is that our method additionally provides the following pointwise relations 
\begin{equation}\label{eq:keydif}
\rho_{f}^{(\ve)} \le \rho_f, \quad \quad |u_{f}^{(\ve)}|\le |u_{f}| .
\end{equation}
These relations are crucially used when applying Lemma \ref{lem:tight}, and also in the derivation of the kinetic entropy inequality later in Section \ref{sec:kei}.

The goal of this section is to prove the following.
\begin{proposition}
\label{prop : reg bgk}
    For each $\ve>0$ there exists $f^\ve\in L^\infty([0,\infty);L^1\cap L^{1+2/d}(\R^{2n}))$ a mild solution to \eqref{eq : reg bgk}, in other words it verifies
    \begin{align}\label{eq: mild form ve}
        f^\ve(t,x,v) = e^{-t} f_0(x-vt,v) + \int_0^t e^{s-t} M^{(\ve)}[f^\ve](s,x-v(t-s),v) \ds. 
    \end{align}
    Moreover, there is a uniform bound on the following quantities
        \begin{equation}
        \label{eq : reg bgk unif bdd}
        \begin{cases}
        \|f^\ve\|_{L^\infty([0,\infty);L^1 \cap L^{1+2/d}(\R^{2n}))} \le C(f_0),\\[8pt]
        \displaystyle \sup_{t \in [0,\infty)} \iint_{\R^{2n}} |v|^2 f^\ve \dx \dv \le C(f_0).
        \end{cases}
    \end{equation}
    Finally, for any $T>0$ the following tightness condition holds for the macroscopic densities:
    \begin{equation}
    \label{eq : tightness}
    \begin{split}
        &\lim_{R\to \infty} \sup_{0\le t \le T} \sup_{0<\ve<1} \int_{\{|x| > R\}} \rho_{f^\ve}(t,x) \dx = 0.
    \end{split}
    \end{equation}
\end{proposition}

We begin with a standard iterative sequence $\{f^\ve_k\}_{k\ge 0}$ with initial step $f^\ve_0:= f_0$ at $k=0$.
\begin{lemma}
\label{lem : iterative seq}
    For each $\ve >0$ and $k\in \bbN$ there is a unique mild solution $f^\ve_{k+1}\in L^\infty((0,\infty);L^1\cap L^{1+2/d}(\R^{2n}))$ to the system
    \begin{align} \label{eq: lin k}
    \begin{cases}
        \p_t f^\ve_{k+1} + v\cdot \nabla_x f^\ve_{k+1} = M^{(\ve)}[f_k^\ve] - f_{k+1}^\ve,\\
        f_{k+1}^\ve(0,x,v) = f_0(x,v)
    \end{cases}
    \end{align}
    verifying
    \begin{equation}
    \label{eq : k mild form}
    f_{k+1}^\ve(t,x,v) = e^{-t} f_0(x-vt,v) + \int_0^t e^{s-t} M^{(\ve)}[f_k](s,x-v(t-s),v) \ds.
\end{equation}
Moreover there are the uniform-in-$k,\ve$ estimates
\begin{align}
\label{eq : unif in k, ve}
\begin{cases}
    \|f^\ve_{k}\|_{L^\infty([0,\infty);L^1\cap L^{1+2/d}(\R^{2n}))} \le C(f_0),\\
    \|f^\ve_{k}\|_{L^\infty([0,\infty);L^{1}_2(\R^{2n}))} \le C(f_0).
\end{cases}
\end{align}
\end{lemma}

To streamline the proof of uniform estimates, we present the following auxiliary lemma.

\begin{lemma}\label{lem:ind}
Let $\{a_k(\cdot)\}_{k=0}^\infty$ be a sequence of $C^1([0,\infty);\R)$ functions satisfying
\[
a_{k}'(t) +a_{k}(t) \le a_{k-1}(t),\quad a_k(0)=a_0(t) = a_0 \in [0,\infty)
\]
for all $t \ge 0$ and $k \in \N$. Then we have $a_k(t) \le a_0$.
\end{lemma}
\begin{proof}
At the base step $k=0$, the assertion of the lemma is trivial by the assumption $a_0(t) = a_0$. Assume now that the hypothesis holds as $a_k(t) \le a_0$. An application of Gr\"onwall's lemma to the given inequality shows that
\begin{align*}
    a_{k+1}(t) &\le e^{-t} a_{k+1}(0) + \int_0^t e^{s-t} a_k(s) \ds \\
    &\le e^{-t} a_{k+1}(0) + a_0 \int_0^t e^{s-t} \ds \\
    &= e^{-t} a_0 + a_0 (1 - e^{-t}) = a_0,
\end{align*}
which completes the induction.
\end{proof}

\begin{proof}[Proof of Lemma \ref{lem : iterative seq}]
Since $\rho_{f}^{(\ve)}, u_{f}^{(\ve)} \le 1/\ve$ holds for any $f$, we have $M^{(\ve)}[\cdot]\in L^\infty([0,\infty);L^1\cap L^\infty(\R^{2n}))$. The existence and uniqueness of each $f^\ve_k \in L^\infty([0,\infty);L^1\cap L^\infty(\R^{2n}))$ which verifies \eqref{eq : k mild form} then follows by standard results for linear transport equations. Assuming a fixed value of $\ve$ let us write $f_k := f^\ve_k$ to lighten the notation.

Since
\begin{align*}
    \frac{\textnormal{d}}{\textnormal{d}t}\iint_{\R^{2n}} f_{k+1} \dx \dv + \iint_{\R^{2n}} f_{k+1}\dx\dv &= \iint_{\R^{2n}} M^{(\ve)}[f_k] \dx \dv \\
    &= \int_{\R^n} \rho^{(\ve)}_{f_k} \dx 
    \le \int_{\R^n} \rho_{f_k} \dx
    = \iint_{\R^{2n}} f_k \, \dx \dv ,
\end{align*}
we apply Lemma \ref{lem:ind} to deduce that
\begin{align*}
    \sup_{t\in [0,\infty)} \iint_{\R^{2n}} f_{k+1}\dx \dv \le \iint_{\R^{2n}} f_0\dx\dv,
\end{align*}
which is the desired uniform $L^1$-estimate mentioned in $\eqref{eq : unif in k, ve}_1$.

To obtain the $L^{1+2/d}$-estimate, we test $\frac{d+2}{d}f_{k+1}^{2/d}$ against the linearized equation \eqref{eq: lin k}. This, and an application of Young's inequality, show
\begin{align*}
    \frac{\textnormal{d}}{\textnormal{d}t} \iint_{\R^{2n}} f_{k+1}^{1+2/d} \dx\dv
    &= \frac{d+2}{d}\iint_{\R^{2n}} M^{(\ve)}[f_k] \, f_{k+1}^{2/d} \dx\dv - \frac{d+2}{d}\iint_{\R^{2n}} f_{k+1}^{1+2/d}\dx\dv \\
    &\le \frac{d+2}{d}\iint_{\R^{2n}} \left(\frac{d}{d+2}(M^{(\ve)}[f_k])^{1+2/d} + \frac{2}{d+2} f_{k+1}^{1+2/d} \right)\dx\dv\\
    &\quad - \frac{d+2}{d}\iint_{\R^{2n}} f_{k+1}^{1+2/d}\dx\dv\\
    &= \iint_{\R^{2n}} (M^{(\ve)}[f_k])^{1+2/d} \dx\dv - \iint_{\R^{2n}} f_{k+1}^{1+2/d}\dx\dv.
\end{align*}
Hence, we obtain
\begin{align*}
    \frac{\textnormal{d}}{\textnormal{d}t}\iint_{\R^{2n}} f_{k+1}^{1+2/d} \dx \dv + \iint_{\R^{2n}} f_{k+1}^{1+2/d} \dx \dv \le \iint_{\R^{2n}} (M^{(\ve)}[f_k])^{1+2/d} \dx \dv.
\end{align*}
On the other hand, it is immediate that
\begin{align*}
    \frac{\textnormal{d}}{\textnormal{d}t}\iint_{\R^{2n}} \frac{|v|^2}{2} f_{k+1} \dx \dv + \iint_{\R^{2n}} \frac{|v|^2}{2} f_{k+1} \dx \dv = \iint_{\R^{2n}} \frac{|v|^2}{2} M^{(\ve)}[f_k] \dx \dv.
\end{align*}
Together we find
\begin{align}
\label{eq : d dt H f_k+1}
    \frac{\textnormal{d}}{\textnormal{d}t}\iint_{\R^{2n}} H(f_{k+1},v) \dx \dv + \iint_{\R^{2n}} H(f_{k+1},v) \dx \dv \le \iint_{\R^{2n}} H(M^{(\ve)}[f_k],v) \dx \dv.
\end{align}
From \eqref{eq: comp cond} and the pointwise relations in \eqref{eq:keydif}
\begin{equation}
\label{eq : comparing M and Mve}
\begin{split}
    \int_{\R^n} H(M^{(\ve)}[f_k],v) \dv &= \frac{1}{2}(\rho_{f_k}^{(\ve)})|u_{f_k}^{(\ve)}|^2 + \frac{1}{\gamma-1}(\rho_{f_k}^{(\ve)})^\gamma  \\
    &\le \frac{1}{2}\rho_{f_k}|u_{f_k}|^2 + \frac{1}{\gamma-1}(\rho_{f_k})^\gamma = \int_{\R^n} H(M[f_k],v)\dv.
    \end{split}
\end{equation}
Using \eqref{eq : comparing M and Mve} and the minimization principle, \eqref{eq : d dt H f_k+1} reduces to
\begin{align} \label{eq: H k est}
    \frac{\textnormal{d}}{\textnormal{d}t}\iint_{\R^{2n}} H(f_{k+1},v) \dx \dv + \iint_{\R^{2n}} H(f_{k+1},v) \dx \dv \le \iint_{\R^{2n}} H(f_k,v) \dx \dv
\end{align}
and an application of Lemma \ref{lem:ind} shows 
\begin{align*}
    \sup_{t\in [0,\infty)} \iint_{\R^{2n}} H(f_{k+1},v)\dx\dv \le \iint_{\R^{2n}} H(f_0,v) \dx \dv .
\end{align*}
By definition this proves the $L^{1+2/d}$-estimate in $\eqref{eq : unif in k, ve}_1$ and also $\eqref{eq : unif in k, ve}_2$.
\end{proof}

\begin{proof}[Proof of Proposition \ref{prop : reg bgk}]
By the same calculations in \cite[Section 3.1.2]{choihwang2024} we verify as a result of \eqref{eq : max f g stab} that the sequence $\{f_k\}$ is Cauchy in $L^\infty_{\rm loc}([0,\infty);L^1_2(\R^{2n}))$. Hence, the $f_k$ converges in this space to a limit $f^\ve \in L^\infty_{\rm loc}([0,\infty);L^1_2(\R^{2n}))$. In view of the uniform in time estimates in \eqref{eq : unif in k, ve}, we deduce that $f^\ve \in L^\infty([0,\infty);L^1_2(\R^{2n}))$.

Also, by applying \eqref{eq : max f g stab}, we obtain the convergence $M^{(\ve)}[f_k]\to M^{(\ve)}[f^\ve]$ in $L^\infty_{\rm loc}([0,\infty);L^1_2(\R^{2n})).$ Indeed, for any $T>0$,
\begin{align*}
    &\sup_{0\le t \le T}\iint_{\R^{2n}}(1+|v|^2)\lt|M^{(\ve)}[f_k] - M^{(\ve)}[f^\ve]\rt| \dx\dv \\
    &\le \sup_{0\le t \le T} C(\ve) \iint_{\R^{2n}} (1+|v|^2) |f_k- f^\ve| \dx\dv \to 0\quad \text{as}\quad k\to\infty.
\end{align*}
The convergence of the Maxwellians implies that the integral arising in the mild form \eqref{eq : k mild form} converges as
\begin{align} \label{eq: conv Mk}
    \int_0^t e^{s-t} M^{(\ve)}[f_k](s,x-v(t-s),v) \ds \to \int_0^t e^{s-t} M^{(\ve)}[f^\ve](s,x-v(t-s),v) \ds 
\end{align}
in $L^\infty_{loc}([0,\infty);L^1_2(\R^{2n}))$. To see this, for each $T>0$ and $t\in [0,T]$, note that
\begin{align*}
& \iint_{\R^{2n}} (1+|v|^2) \left|\int_0^t e^{s-t} (M^{(\ve)}[f_k] - M^{(\ve)}[f^\ve])(s,x-v(t-s),v) \ds  \right|\dx\dv\\
    \le& \int_0^t \iint_{\R^{2n}} e^{s-t} (1+|v|^2) \left|(M^{(\ve)}[f_k] - M^{(\ve)}[f^\ve])(s,x-v(t-s),v)\right| \dx\dv\ds \\
    = & \int_0^t \iint_{\R^{2n}} e^{s-t} (1+|v|^2) \left| M^{(\ve)}[f_k](s,x,v) - M^{(\ve)}[f^\ve](s,x,v) \right| \dx\dv\ds \\
    \le &\lt(\int_0^t e^{s-t}\ds \rt)\left(\sup_{0\le s \le T} \iint_{\R^{2n}} (1+|v|^2)\left| M^{(\ve)}[f_k](s,x,v) - M^{(\ve)}[f^\ve](s,x,v) \right| \dx\dv \right)\\
   \le &\lt\|M^{(\ve)}[f_k] - M^{(\ve)}[f^\ve]\rt\|_{L^\infty(0,T;L^1_2(\R^{2n}))} \to 0 \quad \text{as $k\to\infty$}.
\end{align*}
Thus, we can pass to the limit in \eqref{eq : k mild form} to verify straightforwardly that $f^\ve$ is a mild solution to \eqref{eq : reg bgk}. The estimates asserted in \eqref{eq : reg bgk unif bdd} then follow directly from Lemma \ref{lem : iterative seq}.

Lastly it remains to establish $\eqref{eq : tightness}$. It follows as a direct consequence of Lemma \ref{lem:tight} below. Note that the condition \eqref{eq : collision sign} holds with $G^\ve = M^{(\ve)}[f^\ve] - f^\ve$ due to the relation $\rho_{f^\ve}^{(\ve)} \le \rho_{f^\ve}$ in \eqref{eq:keydif}.
\end{proof}

\begin{lemma}\label{lem:tight}
    Let $\{f^m\} \subset L^\infty([0,\infty);L^1_2(\R^{2n}))$ be a family of weak solutions to the transport equations
    \begin{align}
    \label{eq : transport m}
        \begin{cases}
        \p_t f^m + v\cdot \nabla_x f^m =G^m,\\
        f^m(0,\cdot,\cdot) = f_0(\cdot,\cdot)
        \end{cases}
    \end{align}
    where $G^m \in L^\infty([0,\infty);L^1(\R^{2n}))$ and
    \begin{align}
    \label{eq : collision sign}
        \int_{\R^n} G^m \dv \le 0.
    \end{align}
    Assume also
    \begin{equation*}
        \mathcal{E} := \sup_{m} \sup_{t\in [0,\infty)} \iint_{\R^{2n}} (1+|v|^2)f^m \dx \dv < \infty.
    \end{equation*}
    Then, if $\psi\in C^0(\R^n)$ verifies $|\psi(v)| \lesssim (1+|v|)^\sigma$
    for some $\sigma\in [0,2)$, the function
    \[
    \rho^m_\psi(t,x): = \intr f^m(t,x,v) \psi(v) \dv
    \]
    satisfies the estimate
\[
\sup_{m}\sup_{t \in [0,T]}\int_{|x|\ge 2R} \left|\rho_{\psi}^{m}(t,x)\right| \dx \le \mathcal{E}^{\frac{\sigma}{2}}\lt(\int_{|x|\ge R} \int_{\R^n} f_0 \dv \dx + \frac{2\mathcal{E}T}{R}\rt)^{1-\frac{\sigma}{2}}
\]  
for any $T,R>0$.
\end{lemma}
\begin{proof}
We first demonstrate the case where $\psi(v) \equiv 1$. As will be seen soon, the general case follows by interpolation.
Choose a decreasing function $\xi \in C_c^\infty([0,\infty))$ satisfying $\|\xi'(r)\|_\infty \le 2$ and
\begin{align*}
\begin{cases}
    \xi(r) = 1  &  r \in [0,1],\\
    \xi(r) = 0  &   r \in [2,\infty),
    \end{cases}
\end{align*}
and set $\zeta(r):= 1- \xi(r)$. Then, for each $R>0$, define $\chi_R:\R^n\to\R$ with $\chi_R(x) = \zeta\left(\frac{|x|}{R}\right)$. Similarly, we set $\eta_k(\cdot):= \xi\lt(\frac{|\cdot|}{k}\rt)$ for $k>0$.

For each $k,l>R$, we use $\chi_R(x)\eta_k(x)\eta_l(v) \in C_c^\infty(\R^{2n})$ as a test function in the weak formulation of \eqref{eq : transport m}: for any $t \in [0,T]$,
\begin{equation}\label{eq: form ftc}
\begin{split}
    &\iint_{\R^{2n}} f^m \chi_R(x)\eta_k(x)\eta_l(v)\dv\dx  \\
    &= \iint_{\R^{2n}} f_0 \chi_R(x)\eta_k(x)\eta_l(v)\dv\dx \\
    &\quad +  \int_0^t\iint_{\R^{2n}} v f^m \eta_l(v) \cdot \nabla_x (\chi_R(x)\eta_k(x)) \dv\dx \ds\\
    &\quad + \int_0^t\iint_{\R^{2n}} G^m  \chi_R(x)\eta_k(x) \eta_l(v)\dx\dv\ds \\
    &=: I + II + III.
\end{split}
\end{equation}
Since $0 \le \eta_k(\cdot), \eta_l(\cdot) \le 1$, we first have that
\[
I \le \iint_{\R^{2n}} f_0 \chi_R(x)\dv\dx .
\]
The second term is majorized as
\[
II \le \frac{2}{R}\int_0^t \iint_{\R^{2n}} (1+|v|^2)f^m \dv\dx \le \frac{2\calE T}{R},
\]
thanks to the relations $\|\nabla_x\chi_R\|_{\infty} \le \frac{2}{R}$, $\|\nabla_x\eta_k\|_\infty \le \frac{2}{k}\le \frac{2}{R}$, and $|v| \le 1 + |v|^2$. For the last term, since $G^m \in L^\infty([0,\infty);L^1(\R^{2n}))$ and $\eta_l \nearrow 1$, we find from the dominated convergence theorem that:
\[
\lim_{l \to \infty} III =  \int_0^t\iint_{\R^{2n}}  G^m \chi_R(x)\eta_k(x)\dx\dv\ds \le 0,
\]
where the inequality follows from  \eqref{eq : collision sign} and non-negativity of $\chi_R(x) \eta_k(x)$.

For the left-hand side of \eqref{eq: form ftc}, we note that the monotone convergence theorem provides
\[
\lim_{k,l \to \infty} \iint_{\R^{2n}} f^m \chi_R(x)\eta_k(x)\eta_l(v)\dv\dx = \iint_{\R^{2n}} f^m \chi_R(x) \dv\dx.
\]
Thus, taking the limsup $k,l \to \infty$ in \eqref{eq: form ftc}, we obtain
\[
\iint_{\R^{2n}} f^m \chi_R(x) \dv\dx \le \iint_{\R^{2n}} f_0 \chi_R(x)\dv\dx + \frac{2\calE T}{R}.
\]
Finally, we employ the following relation
\begin{align*}
        \mathbf{1}_{|x|\ge 2R }(x) \le \chi_R(x) \le \mathbf{1}_{|x|\ge R} (x)
\end{align*}
to obtain
\begin{equation}\label{eq : mass tight}
\begin{split}
    \int_{|x|\ge 2R} \rho_{f^m} \dx \le \int_{|x|\ge R} \rho_{f_0}\dx + \frac{2\calE T}{R}.
    \end{split}
\end{equation}
The general case regarding $\rho_{\psi}^m$ with $\sigma \in [0,2)$ follows then by interpolation. Indeed applying H\"older's inequality we deduce that 
\begin{align*}
\int_{|x|\ge 2R} \left|\rho_{\psi}^{m} \right| \dx 
&\le  \lt(\iint_{\R^{2n}} (1+|v|^2) f^m \dv\dx\rt)^{\frac{\sigma}{2}} \lt(\int_{|x|\ge 2R} \rho_{f^{m}} \dx\rt)^{1- \frac{\sigma}{2}}  \\
&= \mathcal{E}^{\frac{\sigma}{2}} \lt(\int_{|x|\ge 2R} \rho_{f^{m}} \dx\rt)^{1- \frac{\sigma}{2}} 
\end{align*}
which combined with \eqref{eq : mass tight} completes the proof.
\end{proof}

\section{Proof of Theorem \ref{thm : main theorem 1} : Convergence Analysis} \label{sec:pss}

To obtain compactness of the macroscopic variables corresponding to $f^\ve$ we introduce a version of the velocity averaging lemma, reminiscent of \cite[Lemma 2.7]{karpermellettrivisa2013}. Here the given sequence is not assumed to be bounded in $L^\infty$ and more importantly the assumption on the spatial moment is replaced by a tightness condition.
\begin{proposition}
\label{Prop : Vel Avg Lem}
    Let $p\in (1,\infty)$, and $f^m$ solutions to the transport equations
    \begin{align*}
        \begin{cases}
        \p_t f^m + v \cdot \nabla_x f^m = G^m_+ - G^m_-, \\f^m(0,\cdot,\cdot) = f_0(\cdot,\cdot) \in L^p(\R^{2n})
        \end{cases}
    \end{align*}
    where $G^m_+, G^m_- \ge 0$. Assume the following:
    \begin{enumerate}
    \item $\{f^m\}$ is bounded in $L^\infty([0,\infty);L^p(\R^{2n}))$.
    \item $\{(1+|v|^2)f^m\}$ is bounded in $L^\infty([0,\infty);L^1(\R^{2n}))$.
    \item $\{G^m_+\}$ and $\{G^m_-\}$ are bounded in $L_{loc}^1([0,\infty);L^1(\R^{2n}))$.
    \item For each $T>0$, the family $\{\rho_{f^m}(t,\cdot)\}_{(m,t)\in\N\times[0,T]}$ is tight:
    \begin{align*}
    \lim_{R\to \infty} \sup_{0\le t \le T} \sup_{m} \int_{|x|\ge R} \rho_{f^m} \dx = 0.
    \end{align*}
    \end{enumerate}
    Then, for each $\psi \in C(\R^n)$ verifying $|\psi(v)| \lesssim (1+|v|)^\sigma$ with $\sigma\in [0,2)$, the sequence $\left\{\int_{\R^n}f^{m}\psi(v)\dv\right\}$ is relatively compact in $L_{loc}^1([0,\infty);L^1(\R^n))$.
\end{proposition}
\begin{proof}
We aim to prove that for each $\eta>0$ and $T>0$, it is possible to extract a subsequence (depending on $\eta$ and $T$) of
\[
\rho_{\psi}^m := \int_{\R^n} f^m(t,x,v) \psi(v) \dv
\]
satisfying
\begin{equation}\label{eq:Cauchy}
\limsup_{k,\ell \to \infty} \|\rho_{\psi}^{m_k} - \rho_{\psi}^{m_l}\|_{L^1([0,T]\times \R^n)} \le \eta.
\end{equation}
Suppose that we have proven this. By taking 
\[
\eta_j := \frac{1}{j} \to 0 \quad \text{and} \quad T_j:=j \to \infty,
\]
we can extract subsequences verifying \eqref{eq:Cauchy} recursively, then use a standard diagonal argument to finally pass to a subsequence which is Cauchy in $L^1([0,T]\times \R^n)$ for any $T>0$. The conclusion of the proposition would then follow, and hence it suffices to prove the above claim.

Let $\eta>0$ and $T>0$ be given. By interpolating the assumptions (2) and (4) as we did in Lemma \ref{lem:tight} we can fix $R>0$ satisfying
    \begin{equation}\label{eq:farfield}
    \sup_{m\in \mathbb{N}} \int_0^T\int_{|x|\ge R} \left| \rho_{\psi}^m \right| \dx  \dt \le \frac{\eta}{8}.
    \end{equation}
Consider next smooth cutoff functions
    \begin{align*}
        \Xi_r(v) = \begin{cases}
            1 & |v| \le r,\\
            0 & |v| \ge r+1
        \end{cases}
    \end{align*}
    and set
    \begin{align*}
        \rho_{\psi}^{m,r} := \int_{\R^n} f^m(t,x,v) \psi(v) \Xi_r(v) \dv.
    \end{align*}
By classical velocity averaging results \cite[Section III]{dipernalionsmeyer1991}, for each $\phi \in C_c^\infty(\R^n)$ the sequence $\left\{\int_{\R^n} f^m \phi(v)\dv\right\}$ is relatively compact in $L^1([0,T]\times B_R(0))$. As its direct application, for each $r>0$ there is a subsequence of $\{\rho_\psi^{m,r}\}_m$ (depending on $r$) that is convergent in $L^1([0,T]\times B_R(0))$, and for this subsequence clearly
\begin{equation*}
\limsup_{k,l\to \infty} \|\rho_\psi^{m_k,r} - \rho_\psi^{m_l,r}\|_{L^1([0,T]\times B_R(0))} \le \frac{\eta}{4}.
\end{equation*}
Thanks to \eqref{eq:farfield}, we deduce
\begin{align*}
    \limsup_{k,l\to\infty} \|\rho_\psi^{m_k,r} - \rho_\psi^{m_l,r}\|_{L^1([0,T]\times \R^n)}  \le \frac{\eta}{2}.
\end{align*}
On the other hand, we note that for any $r>0$
\begin{align*}
\intr |\rho_\psi^m - \rho_\psi^{m,r}| \, \dx &\le  \int_{\R^n}\int_{|v|\ge r} |f^m(t,x,v)\psi(v)| \, \dv \dx  \\
&\le \frac{C}{(1+r)^{2-\sigma}}\iint_{\R^{2n}} (1+|v|)^2 f^m \dv\dx.
\end{align*} 
Therefore, in view of the assumption (2) we can fix large $r = r(T) >0$ verifying
\[
\sup_{m\in \mathbb{N}} \|\rho_\psi^m - \rho_\psi^{m,r}\|_{L^1([0,T]\times \R^n)} \le \frac{\eta}{4}.
\]
Finally, along the subsequence corresponding to this fixed $r$
\begin{align*}
    &\limsup_{k,l\to\infty}\|\rho_{\psi}^{m_k} - \rho_{\psi}^{m_l}\|_{L^1([0,T]\times \R^n)}\\
    &\le \limsup_{k,l\to\infty} \Big(\|\rho_\psi^{m_k} - \rho_{\psi}^{m_k,r}\|_{L^1([0,T]\times \R^n)} + \|\rho_{\psi}^{m_k,r} - \rho_\psi^{m_l,r}\|_{L^1([0,T]\times \R^n)}\\
    &\qquad \qquad \qquad + \|\rho_\psi^{m_l,r} - \rho_\psi^{m_l}\|_{L^1([0,T]\times\R^n)}\Big)\\
    &\le \eta
\end{align*}
and hence \eqref{eq:Cauchy} is proven. This completes the proof.
\end{proof}

Now, Proposition \ref{prop : reg bgk} and the Dunford--Pettis theorem imply that there is a limit $f\in L^\infty([0,\infty); L^1(\R^{2n}))$ with
\begin{align*}
    f^\ve \weakto f \quad \text{in} \quad L^{1}_{\rm loc}([0,\infty);L^1(\R^{2n})).
\end{align*}
As a direct application of Proposition \ref{Prop : Vel Avg Lem} to the sequence $\{f^\ve\}$ there is a subsequence verifying
\begin{align*}
    \rho_{f^\ve} \to \rho_f, \quad \rho_{f^\ve}u_{f^\ve}\to \rho_f u_f \quad\text{strongly in}\quad L^1_{loc}([0,\infty);L^1(\R^n)).
\end{align*}
We can pass to a further subsequence and assume that the convergences above also hold a.e. in $[0,\infty)\times \R^n$. This moreover implies
\begin{align*}
    &\rho_{f^\ve}^{(\ve)} = \frac{\rho_{f^\ve}}{1 + \ve \rho_{f^\ve}} \to \rho_f \quad \text{a.e. in }[0,\infty)\times \R^n,\\
    &u_{f^\ve}^{(\ve)} = \frac{\rho_{f^\ve}u_{f^\ve}}{\rho_{f^\ve} + \ve(1+|\rho_{f^\ve}u_{f^\ve}|)} \to u_f \quad \text{a.e. in }\{(t,x)\in [0,\infty)\times\R^n :\rho_f(t,x)>0\}.
\end{align*}
As a consequence we deduce the strong convergence of the Maxwellians.
\begin{lemma}\label{lem:Msc}
    It holds that
    \begin{align*}
        M^{(\ve)}[f^\ve] \to M[f] \quad\text{in}\quad L^1_{loc}([0,\infty);L^1(\R^{2n})).
    \end{align*}
\end{lemma}
\begin{proof}
    For any $T>0$, we write
    \begin{align*}
        &\iiint_{[0,T]\times\R^{2n}} |M^{(\ve)}[f^\ve] - M[f]|\dx \dt \dv \\
        &\qquad = \left(\iiint_{\{\rho_f>0\}\times \R^n} + \iiint_{\{\rho_f = 0\}\times \R^n}\right) |M^{(\ve)}[f^\ve] - M[f]|\dx \dt \dv\\
        &\qquad = I^\ve + J^\ve.
    \end{align*}
    On the set $\{\rho_f > 0\}\subset [0,T]\times\R^n$ we have the pointwise convergence $M^{(\ve)}[f^\ve] \to M[f]$. Recall that the family $\{M^{(\ve)}[f^\ve]\}$ is bounded in both $L^\infty([0,\infty);L^1(\R^{2n}))$ and $L^\infty([0,\infty);L^{1+2/d}(\R^{2n}))$. Also, since
    \begin{align*}
        \int_{|x|\ge R}\int_{|v|\ge R} M^{(\ve)}[f^\ve] \dv \dx \le \int_{|x|\ge R}\intr M^{(\ve)}[f^\ve] \dv \dx &= \int_{|x|\ge R} \rho_{f^\ve}^{(\ve)} \dx\\
        &\le \int_{|x|\ge R} \rho_{f^\ve}\dx, 
    \end{align*}
    the family $\{M^{(\ve)}[f^\ve]\}$ is tight thanks to \eqref{eq : tightness}. Altogether we deduce that the Vitali convergence theorem is applicable and $I^\ve\to 0$ as $\ve \to 0$.

    For $J^\ve$ we simply note that $M[f] = 0$ on the set $\{\rho_f = 0\}$, hence
    \begin{align*}
        J^\ve = \iiint_{\{\rho_f = 0\}\times \R^n} M^{(\ve)}[f^\ve] \dx \dv \dt = \iint_{\{\rho_f = 0\}} \rho_{f^\ve}^{(\ve)} \dx \dt &\le \iint_{\{\rho_f = 0\}} \rho_{f^\ve} \dx \dt\\
        &\to \iint_{\{\rho_f = 0\}} \rho_f \dx \dt = 0
    \end{align*}
    thanks to the strong convergence $\rho_{f^\ve}\to \rho_f$ in $L^1([0,T]\times\R^n)$.  
\end{proof}

We then obtain strong convergence of the sequence $\{f^\ve\}$ itself.

\begin{lemma}
\label{lem : conv of fve}
    It holds that
    \begin{align*}
        f^\ve \to f \quad\text{in}\quad L^\infty_{\rm loc}([0,\infty);L^1(\R^{2n})).
    \end{align*}
    In particular, the limit $f$ verifies for any $t\in [0,\infty)$
    \begin{align*}
        f(t,x,v) = e^{-t} f_0(x-vt,v) + \int_0^t e^{s-t}M[f](t,x-v(t-s),v) \ds
    \end{align*}
    and is thus a mild solution to \eqref{BGK} in the sense of Definition \ref{def:mild}.
\end{lemma}
\begin{proof}
Thanks to the strong convergence of the Maxwellians verified in the previous lemma, we readily verify that
\begin{align*}
    \int_0^t e^{s-t} M^{(\ve)}[f^\ve](s,x-v(t-s),v)\ds \to \int_0^t e^{s-t} M[f](s,x-v(t-s),v)\ds 
\end{align*}
in $L^\infty_{loc}([0,\infty);L^1(\R^{2n}))$. Indeed:
\begin{align*}
    &\sup_{0\le t \le T} \iint_{\R^{2n}}\left|\int_0^t e^{s-t} (M^{(\ve)}[f^\ve] - M[f])(s,x-v(t-s),v) \ds\right| \dx\dv \\
    &\le \sup_{0\le t \le T} \int_0^t \iint_{\R^{2n}}  |M^{(\ve)}[f^\ve] - M[f]|(s,x-v(t-s),v) \dx\dv\ds \\
    &\le \int_0^T \iint_{\R^{2n}}  |M^{(\ve)}[f^\ve] - M[f]|(s,x,v) \dx\dv\ds \to 0 \quad \forall \, T>0.
\end{align*}
Hence, passing to the limit in \eqref{eq: mild form ve}, we obtain
\begin{align*}
    f^\ve (t,x,v) &\to e^{-t}f_0(x-vt,v) + \int_0^t e^{s-t} M[f](s,x-v(t-s),v) \ds\\
    &\quad = f(t,x,v) \quad \text{in} \quad L^\infty(0,T;L^1(\R^{2n})) \quad \forall \, T>0.
\end{align*}
As this strong convergence holds for any $T>0$, we recall the uniform in time estimates provided in \eqref{eq : reg bgk unif bdd} to deduce that $f \in L^\infty([0,\infty);L^1_2 \cap L^{1+2/d} (\R^{2n}))$. It is then straightforward that $f$ is a mild solution to the BGK equation \eqref{BGK} in the sense of Definition \ref{def:mild}.
\end{proof}

\section{Proof of Theorem \ref{thm : main theorem 1} : Kinetic Entropy Inequality} \label{sec:kei} 

To complete the proof of Theorem \ref{thm : main theorem 1}, it only remains to show that the kinetic entropy inequality \eqref{eq : kin entrop ineq} holds.\\
We need to recover previous estimates from the linearized level. Namely from \eqref{eq : d dt H f_k+1}
\begin{align*}
    \frac{\textnormal{d}}{\textnormal{d}t}\iint_{\R^{2n}}H(f_{k+1},v)\dx\dv + \iint_{\R^{2n}} H(f_{k+1},v)\dx\dv \le \iint_{\R^{2n}} H(M^{(\ve)}[f_k],v)\dx\dv.
\end{align*}
Hence,
\begin{equation}
\begin{split}
\label{eq : kin entrop ineq, lin level}
    \iint_{\R^{2n}} H(f_{k+1},v)\dx \dv &+ \int_0^t \iint_{\R^{2n}} H(f_{k+1},v)-H(M^{(\ve)}[f_k],v) \dx \dv \\
    &\qquad \le \iint_{\R^{2n}}H(f_0,v)\dx\dv. 
    \end{split}
\end{equation}
Recall that $f_{k}\to f^\ve$ strongly in $L^\infty(0,T;L^1_2(\R^{2n}))$ for any $T>0$. In view of the uniform estimates in \eqref{eq : unif in k, ve} this also implies that for each $t\in [0,T]$ the sequence $\{f_k(t)\}$ converges weakly in $L^{1+2/d}(\R^{2n})$ to $f^\ve(t)$. By the strong convergence in $L^1_2(\R^{2n})$ and lower semicontinuity of the norm under weak convergence, we deduce that for all $t\in[0,T]$
\begin{align*}
    \iint_{\R^{2n}} H(f^\ve,v)\dx\dv \le \liminf_{k\to\infty} \iint_{\R^{2n}} H(f_{k+1},v)\dx\dv.
\end{align*}
For the Maxwellian, we use \eqref{eq: comp cond} and \eqref{eq : diff of reg quantities} to note that for any $g,h\in L^1_2(\R^{2n})$ it holds:
\begin{align*}
    &\left|\iint_{\R^{2n}} H(M^{(\ve)}[g],v) - H(M^{(\ve)}[h],v) \,\dx\dv \right|\\
    &= \left|\int_{\R^n} \frac{1}{2}\rho^{(\ve)}_{g}|u^{(\ve)}_{g}|^2 - \frac{1}{2}\rho_{h}^{(\ve)}|u_{h}^{(\ve)}|^2 + \frac{1}{\gamma-1}(\rho_{g}^{(\ve)})^\gamma - \frac{1}{\gamma-1}(\rho_{h}^{(\ve)})^\gamma \, \dx \right| \\
    &= \left|\int_{\R^n} \frac{1}{2}(\rho_{g}^{(\ve)} - \rho_{h}^{(\ve)}) |u_{g}^{(\ve)}|^2 + \rho_{h}^{(\ve)} (u_{g}^{(\ve)} + u_{h}^{(\ve)})\cdot (u_{g}^{(\ve)} - u_{h}^{(\ve)}) \, \dx \right|\\
    &\qquad + \frac{1}{\gamma-1}\left|\int_{\R^n} (\rho_{g}^{(\ve)})^\gamma - (\rho_{h}^{(\ve)})^\gamma \, \dx \right|\\
    &\le \frac{1}{2\ve^2}\int_{\R^n} |\rho_g^{(\ve)} - \rho_h^{(\ve)}| \dx + \frac{2}{\ve^2}\int_{\R^n} |u_g^{(\ve)} - u_h^{(\ve)}| \dx + \frac{C\gamma}{\ve^{\gamma-1}(\gamma-1)} \int_{\R^n} |\rho_g^{(\ve)} - \rho_h^{(\ve)}| \dx \\
    &\le \left(\frac{1}{2\ve^2} + \frac{2}{\ve^4} + \frac{C\gamma}{\ve^{\gamma-1}(\gamma-1)}\right) \int_{\R^n} |\rho_g - \rho_h| \dx + \frac{2}{\ve^3}\int_{\R^n} |\rho_g u_g - \rho_h u_h| \dx \\
    &\le C(\ve,\gamma) \iint_{\R^n} (1+|v|^2) |g - h| \,\dx\dv.
\end{align*}
Above, the integral at the third line was estimated using the upper bounds in \eqref{eq : diff of reg quantities}: $\rho_g^{(\ve)},\rho_h^{(\ve)},|u_g^{(\ve)}|,|u_h^{(\ve)}|\le 1/\ve$. The integral at the fourth line was estimated by first applying the mean-value theorem to the function $z\mapsto z^\gamma$, then using again the upper bounds for $\rho_g^{(\ve)}$ and $\rho_h^{(\ve)}$. The sixth line is a consequence of the stability estimates in \eqref{eq : diff of reg quantities}.

Consequently, the convergence $f_k\to f^\ve$ in $L^\infty(0,T;L^1_2(\R^{2n}))$ implies
\begin{align*}
    \sup_{0\le t \le T} \left|\iint_{\R^{2n}} H(M^{(\ve)}[f_k],v)\dx\dv - \iint_{\R^{2n}} H(M^{(\ve)}[f^\ve],v)\dx\dv \right| \to 0 \quad\text{as}\quad k\to\infty.
\end{align*}
Together, we can then deduce that \eqref{eq : kin entrop ineq, lin level} implies for all $t\in [0,T]$
\begin{equation}
\label{eq : kin entrop ineq reg level}
\begin{split}
    &\iint_{\R^{2n}} H(f^\ve,v)\dx\dv + \int_0^t \iint_{\R^{2n}} H(f^\ve,v)-H(M^{(\ve)}[f^\ve],v)\dx\dv\\
    &\quad \le \iint_{\R^{2n}} H(f_0,v)\dx\dv.
    \end{split}
\end{equation}
We need now pass to the limit $\ve\to 0$. First, note that Proposition \ref{prop : reg bgk} and the Dunford Pettis theorem imply $f^\ve(t) \weakto f(t)$ weakly in $L^1(\R^{2n})$ for all $t\in [0,T]$, since the limit is already identified by Lemma \ref{lem : conv of fve}. This implies that for any $R>0$
\begin{equation*}
\begin{split}
    \iint_{\R^{2n}} |v|^2 \mathbf{1}_{|v|\le R} f\, \dx \dv &= \lim_{\ve \to 0} \iint_{\R^{2n}} |v|^2 \mathbf{1}_{|v|\le R} f^\ve(t) \dx \dv \\
    &\le \liminf_{\ve \to 0} \iint_{\R^{2n}} |v|^2 f^\ve(t) \dx \dv.
    \end{split}
\end{equation*}
Since this holds for all $R>0$ applying the monotone convergence theorem to the left-hand-side yields
\begin{align*}
    \iint_{\R^{2n}}|v|^2 f(t)\, \dx \dv \le \liminf_{\ve \to 0} \iint_{\R^{2n}} |v|^2 f^\ve(t) \dx \dv.
\end{align*}
Also, we have that $f^\ve(t) \weakto f(t)$ weakly in $L^{1+2/d}(\R^{2n})$. Thus the lower semicontinuity of the norm under weak convergence implies
\begin{equation*}
    \iint_{\R^{2n}} f^{\frac{d+2}{d}}\dx\dv \le \liminf_{\ve\to0}\iint_{\R^{2n}} (f^\ve)^{\frac{d+2}{d}}\dx\dv.
\end{equation*}
Hence
\begin{equation}
\label{eq : kin entrop lsc}
    \iint_{\R^{2n}} H(f,v)\dx\dv \le \liminf_{\ve\to0}\iint_{\R^{2n}} H(f^\ve,v)\dx\dv.
\end{equation}
Second, let us show that
\begin{equation}
\label{eq : diss lsc}
\begin{split}
    &\int_0^t \iint_{\R^{2n}} H(f,v)-H(M[f],v) \dx \dv \ds\\
    &\le \liminf_{\ve\to 0} \int_0^t \iint_{\R^{2n}} H(f^\ve,v) - H(M^{(\ve)}[f^\ve],v) \dx \dv \ds.
\end{split}
\end{equation}
Note as a preliminary step that for any $\delta>0$, thanks to the pointwise convergences $\rho_{f^\ve} \to \rho_f$ and $\rho_{f^\ve}u_{f^\ve} \to \rho_f u_f$
\begin{equation}
\label{eq : ptwise conv HMe}
\begin{split}
    \int_{\R^n} H(M[f^\ve],v)\dv &= \frac{1}{2} \rho_{f^\ve} |u_{f^\ve}|^2 + \frac{1}{\gamma-1}(\rho_{f^\ve})^\gamma \\
    &\to \frac{1}{2}\rho_f |u_f|^2 + \frac{1}{\gamma-1}(\rho_f)^\gamma \\
    &= \int_{\R^n} H(M[f],v)\dv \quad\text{a.e. in }\{(t,x):\rho_f(t,x)>\delta\} \subset [0,T]\times \R^n.
    \end{split}
\end{equation}
From \eqref{eq: comp cond} and \eqref{eq:keydif} we recall again that $\int_{\R^n} H(M^{(\ve)}[f^\ve],v)\dv \le \int_{\R^n} H(M[f^\ve],v)\dv$. Henceforth for any $t\in [0,T]$ and with $A_\delta := \{(s,x):\rho_f(s,x)>\delta\}\subset [0,t]\times \T^n$,
{\allowdisplaybreaks
\begin{align*}
    &\liminf_{\ve \to 0} \int_0^t \intr\intr H(f^\ve,v) - H(M^{(\ve)}[f^\ve],v) \dv \dx \ds\\
    &\ge \liminf_{\ve \to 0} \int_0^t \intr\lt(\intr H(f^\ve,v) - H(M[f^\ve],v) \dv\rt) \dx \ds\\
    &\ge \iint_{A_\delta} \liminf_{\ve \to 0} \left( \int_{\R^n} H(f^\ve,v) - H(M[f^\ve],v) \dv \right) \dx \ds \\
    &= \iint_{A_\delta} \left( \liminf_{\ve \to 0} \int_{\R^n} H(f^\ve,v) \dv  - \int_{\R^n} H(M[f],v) \dv \right) \dx \ds  \\
    &\ge \iint_{A_\delta} \left( \int_{\R^n} H(f,v) - H(M[f],v) \dv\right) \dx \ds.
\end{align*}}
Indeed at the second line above the integrand is non-negative thanks to the minimization principle. Thus the inequality at the third line holds by monotonicity of the integral and then Fatou's lemma. The equality following it is due to \eqref{eq : ptwise conv HMe}. The last inequality follows by Fatou's Lemma since by Lemma \ref{lem : conv of fve} we may assume $f^\ve\to f$ pointwise.

Since this holds for all $\delta>0$ we can let $\delta \to 0$ by applying the dominated convergence theorem. This yields
\begin{equation}
\label{eq : rho>0}
\begin{split}
    &\iint_{\{\rho_f>0\}}\int_{\R^n} H(f,v)-H(M[f],v) \dv \dx \ds \\
    &\quad \le \liminf_{\ve \to 0} \int_0^t \iint_{\R^{2n}} H(f^\ve,v)-H(M^{(\ve)}[f^\ve],v) \dx\dv\ds,
    \end{split}
\end{equation}
where we have denoted $\{\rho_f>0\} := \{(s,x):\rho_f(s,x)>0\}\subset [0,t]\times\R^n$. On the other hand it is clear that $f=0$ a.e. on $\{\rho_f = 0\}\times \R^n$. This and the minimization principle show
\begin{align*}
    0 = \iint_{\{\rho_f = 0\}} \int_{\R^n} H(f,v) \dv \dx \ds = \iint_{\{\rho_f = 0\}} \int_{\R^n} H(M[f],v) \dv \dx \ds,
\end{align*}
and we conclude that \eqref{eq : rho>0} implies \eqref{eq : diss lsc}. Altogether \eqref{eq : kin entrop lsc} and \eqref{eq : diss lsc} show that in the limit $\ve \to 0$, \eqref{eq : kin entrop ineq reg level} yields
\begin{align*}
    \iint_{\R^{2n}} H(f,v) \dx \dv + \int_0^t \iint_{\R^{2n}} H(f,v) - H(M[f],v)\dx\dv\ds \le \iint_{\R^{2n}} H(f_0,v)\dx \dv
\end{align*}
for all $t\in [0,T]$. As $T>0$ is arbitrary we obtain precisely \eqref{eq : kin entrop ineq} for the case $\tau=1$. Thus $f$ is indeed a mild entropy solution to \eqref{BGK}, and this completes the proof of Theorem \ref{thm : main theorem 1}.

\appendix

\section{Remarks on the Minimization Principle for the End-Point Case}\label{rem:countex}

In this section we discuss the minimization principle for the end-point case $\gamma = \frac{n+2}{n}$, focusing on why the infinity term in \eqref{eq : kin entrop ineq} is necessary. First, we make the following observation. As $\gamma \to \frac{n+2}{n}{-}$, we have $d \to {0+}$. The kinetic entropy in \eqref{eq:kin:ent}${}_1$ is then seen to converge formally to
\[
H(f,v):= \frac{1}{2}|v|^2f + \infty \cdot {\bf1}_{f>c_2},
\]
which coincides with $\eqref{eq:kin:ent}_2$.

To elaborate on this matter, we give our interpretation as to why the Maxwellian should have the form that it does. Note that the minimization principle can be understood from a probabilistic point of view. By normalizing $f$, we may assume $\rho_f = 1$ so that $f$ can be understood as a probability measure on $\R^n_v$. Then $u_f$ can be seen as the corresponding expectation, so we may translate $f$ suitably so that $u_f =0$. Similarly the quantity
\[
H(f,v) = \frac{1}{2}|v|^2 f
\]
can be interpreted as the variance corresponding to $f$. Characterizing a minimizer of this quantity thus is to find a centered probability distribution with minimal variance under a constraint that the density is bounded by $c_2$. For any $f$, by rearranging the distribution $f$ to a radial one, this problem need only be addressed in a one-dimensional setting. Hence, we can easily deduce that the minimizer must be a radial characteristic function with height $c_2$, which is precisely the definition of the Maxwellian for the case $\gamma=\frac{n+2}{n}$. This argument easily generalizes to any $\rho_f$ and $u_f$ as well. 

Indeed, the infinity term in \eqref{eq:kin:ent} is necessary for the minimization principle to hold. For instance, if we choose
\[
f(v):= a \cdot {\bf 1}_{|v| \le r}
\]
with $a>c_2$ and $r>0$, then obviously
\[
\rho_f = ar^n\frac{\pi^{n/2}}{\Gamma(\frac{n}{2}+1)},\quad \rho_fu_f=0.
\]
Note that $\frac{\pi^{n/2}}{\Gamma(\frac{n}{2}+1)}$ is the volume of the unit ball in $n$ dimensions. After a direct calculation one easily checks that
\[
\intr |v|^2 f \dv =\lt(\frac{\pi^{n/2}}{\Gamma(\frac{n}{2}+1)}r^{n+2}a\rt) \frac{n}{n+2}.
\]
Now
\[
a > c_2 = \frac{\Gamma(\frac{n}{2}+1)}{(\pi \kappa (n+2))^{n/2}}
\]
so \eqref{eq: comp cond} implies
\begin{equation*}
\begin{split}
\intr |v|^2 M_f \dv =\kappa n \rho_f^{1+\frac{2}{n}} &=  \lt(\frac{\pi^{n/2}}{\Gamma(\frac{n}{2}+1)}r^{n+2}a\rt) \frac{\pi \kappa n a^{2/n}}{\Gamma(\frac{n}{2}+1)^{2/n}} \\
&>\lt(\frac{\pi^{n/2}}{\Gamma(\frac{n}{2}+1)}r^{n+2}a\rt) \frac{n}{n+2},
\end{split}
\end{equation*}
which violates \eqref{eq : min principle} if the infinity term in \eqref{eq:kin:ent} is not present. Note also that this $f\le c_2$ condition\footnote{Actually, the condition $f \le 1$ was imposed in \cite[Proposition 4.1]{berthelinvasseur2005}, but this is simply because our definitions of the Maxwellian differ by a constant factor (see Remark \ref{rem : max const factor}). There is no problem in modifying the arguments of \cite{berthelinvasseur2005} so that with our definition \eqref{def : M}, \cite[Proposition 4.1]{berthelinvasseur2005} holds for $f \le c_2$.} is crucially imposed in \cite[Proposition 4.1]{berthelinvasseur2005}, which is an important step in verifying the hydrodynamic limit.

\section{Proof of Lemma \ref{lem:dis}} \label{sec : proof of lem dis}
We present here a proof of Lemma \ref{lem:dis}. We first recall the following result regarding the Maxwellian of the second model \eqref{eq : ext mod}.
\begin{lemma}\cite[Proposition 4.6]{berthelinvasseur2005}
\label{lem b3} For $\gamma\in\left(1,\frac{n+2}{n}\right)$ and every $\ov{f} = \ov{f}(v,I)$ satisfying
\begin{align*}
\begin{cases}
(1+|v|^2+I^2)\ov{f}\in L^1(\R^n\times[0,\infty);c_0 I^{d-1} \dv \di),\\
0 \le \ov{f} \le c_3,
\end{cases}
\end{align*}
there exists $C=C_{n,\gamma}>0$ such that the following inequality
\begin{equation*}
    \int_{\R^n} \int_0^\infty (|v|^2+I^2) |\overline{f}(v,I) - \overline{M}[\ov{f}](v,I)| c_0 I^{d-1}\di\dv \le C\left\{(\rho_{\overline{f}})^{\frac{\gamma}{2}} \sqrt{\ov{D}_{\ov f}} + \ov{D}_{\ov f} \right\}
\end{equation*}
holds, with 
\[
    \ov{D}_{\ov f} := \displaystyle\int_{\R^n} \int_0^\infty \frac{1}{2}(|v|^2+I^2)(\overline{f}(v,I) - \overline{M}[\ov{f}](v,I)) c_0 I^{d-1} \textnormal{d}I \dv.
\]
\end{lemma}

Our idea is to suitably associate an extra variable $I$ to a function $f=f(v)$, and then apply Lemma \ref{lem b3}. Observe from the definition \eqref{def : M bar} that
\begin{align*}
\overline{M}[\rho,u](v,I) &= c_3 \mathbf{1}_{|v-u|^2 + I^2 < c_1 \rho^{\gamma-1}}\\
&= c_3 {\bf 1}_{I \le (M[\rho,u](v)/c_2)^{1/d}}.
\end{align*}
Motivated by this, for each function $f:\R^n \to \R_+$, we define a corresponding map $\ov{f}:\R^n\times \R_+\to [0,c_3]$ by
\begin{equation}
\label{eq : extending f}
\ov{f}(v,I) := c_3 {\bf{1}}_{I \le (f(v)/c_2)^{1/d}}.
\end{equation}
Then, from the relation $c_0c_3 = dc_2$, we obtain for any $m\ge 0$
\[
\int_0^\infty I^m \ov{f}(v,I) c_0I^{d-1}\di  = dc_2 \int_0^{(f(v)/c_2)^{1/d}} I^{m+d-1}\di = \frac{1}{c_2^{m/d}}\frac{(f(v))^{1+\frac{m}{d}}}{1+m/d}
\]
by the fundamental theorem of calculus. By choosing $m=0$ and $2$, the following relations hold:
\[
\begin{pmatrix} \rho_f \\ \rho_f u_f \end{pmatrix} = \intr  \begin{pmatrix} 1 \\ v \end{pmatrix}f(v) \dv = \intr \begin{pmatrix} 1 \\ v \end{pmatrix}\int_0^\infty \ov{f}(v,I) c_0 I^{d-1}\di \dv
= \begin{pmatrix} \rho_{\ov{f}}  \\ \rho_{\ov{f}} u_{\ov{f}} \end{pmatrix},
\]
\[
M[f](v) = \int_0^\infty \ov{M}[\ov{f}](v,I)c_0I^{d-1}\di,
\]
\[
H(f,v)= \frac{1}{2}|v|^2f + \frac{1}{2c_2^{2/d}}\frac{f^{1+2/d}}{1+2/d} = \int_0^\infty \frac{1}{2}(|v|^2 + I^2) \ov{f}(v,I) c_0I^{d-1}\di.
\]
In particular, we find (recalling the definitions in \eqref{eq : imp term} and Lemma \ref{lem b3})
\begin{align*}
    D_f = \overline{D}_{\ov f}
\end{align*}
whenever $\ov f$ is the extension corresponding to $f$ as constructed in \eqref{eq : extending f}. Thus, owing to Lemma \ref{lem b3},
{\allowdisplaybreaks\begin{align*}
\intr |v|^2 |f-M[f]| \dv
&= \intr |v|^2\lt|\int_0^\infty  \Big( \ov{f}(v,I)-\ov{M}[\ov{f}](v,I) \Big) c_0I^{d-1}\di \rt| \dv \\
&\le \int_{\R^n}\int_0^\infty (|v|^2 + I^2) \lt|\ov{f}(v,I)-\ov{M}[\ov{f}](v,I) \rt| c_0I^{d-1}\di\dv \\
&\le C \lt\{(\rho_{\ov{f}})^{\frac{\gamma}{2}}\sqrt{\ov{D}_{\ov f}} + \ov{D}_{\ov f}\rt\} \\
&= C \lt\{(\rho_{f})^{\frac{\gamma}{2}}\sqrt{D_f} + D_f \rt\}
\end{align*}}
and Lemma \ref{lem:dis} is proven.

\section{Mild to Weak Solutions} \label{app: uniqueness}
This section is devoted to verifying that mild solutions of \eqref{BGK} (Definition \ref{def:mild}) satisfy the weak formulation of the equation as in  Remark \ref{rem: weak}. Precisely, we prove the following statement:

\begin{lemma}
    Let $f \in L^\infty([0,\infty);L^1_2(\R^{2n}))$ denote a mild solution to \eqref{BGK}. Then, for each $T>0$ and $\phi\in C_c^\infty([0,T)\times\R^{2n})$, $f$ verifies
    \begin{align*}
        &-\iint_{\R^{2n}} f_0 \phi_0 \dx\dv - \int_0^T \iint_{\R^{2n}} f(\p_t\phi + v\cdot \nabla_x\phi)\dx\dv\dt\\
        &\quad = \frac{1}{\tau}\int_0^T \iint_{\R^{2n}}(M[f]-f)\phi\, \dx \dv \dt.
    \end{align*}
\end{lemma}

\begin{proof}
For simplicity, we assume $\tau =1$. Multiplying \eqref{defeq : mild} by $\p_t\phi(t,x,v)$ then integrating over $[0,T)\times\R^{2n}$, we use the mild form of $f$ to write
\begin{align*}
    &\int_0^T \iint_{\R^{2n}} f \,\p_t \phi \,\dx\dv\dt\\
    &= \int_0^T \iint_{\R^{2n}} e^{-t} f_0(x-vt,v) \, \p_t\phi(t,x,v) \dx\dv\dt\\
    &\quad + \int_0^T \int_0^t \iint_{\R^{2n}} e^{s-t} \p_t\phi(t,x,v) M[f](s,x-v(t-s),v)\dx\dv\ds\dt\\
    &=: \calJ_1 + \calJ_2.
\end{align*}
$\bullet$ Computation of $\calJ_1$: We make the change of variables $x\mapsto x+vt$ to compute
\begin{align*}
    \calJ_1 &= \int_0^T \iint_{\R^{2n}} e^{-t} f_0(x,v) \p_t\phi(t,x+vt,v)\dx\dv\dt \\
    &= \int_0^T \iint_{\R^{2n}} e^{-t} f_0(x,v) \Big(\p_t (\phi(t,x+vt,v)) - v\cdot \nabla_x\phi(t,x+vt,v) \Big) \dx\dv\dt \\
    &= \int_0^T \iint_{\R^{2n}} e^{-t} f_0(x-vt,v) \phi(t,x,v) \,\dx\dv\dt \\
    &\quad - \iint_{\R^{2n}} f_0(x,v) \phi_0(x,v) \,\dx\dv \\
    &\quad - \int_0^T \iint_{\R^{2n}} e^{-t} f_0(x-vt,v) v\cdot \nabla_x \phi(t,x,v)\,\dx\dv\dt.
\end{align*}
The second equality follows simply from the chain rule: the last equality follows from an integration by parts with respect to $t$, and then another change of variables $x\mapsto x-vt$.\\
$\bullet$ Computation of $\calJ_2$: We make the change of variables $x\mapsto x+v(t-s)$, apply the chain rule, and then use Fubini's theorem to compute
\begin{align*}
    \calJ_2 &= \int_0^T \int_0^t \iint_{\R^{2n}} e^{s-t} \p_t \phi(t,x+v(t-s),v) M[f](s,x,v)\dx\dv\ds\dt \\
    &= \int_0^T \int_0^t \iint_{\R^{2n}} e^{s-t} \p_t(\phi(t,x+v(t-s),v)) M[f](s,x,v)\dx\dv\ds\dt \\
    &\quad - \int_0^T \int_0^t \iint_{\R^{2n}} e^{s-t} v\cdot \nabla_x\phi(t,x+v(t-s),v) M[f](s,x,v)\dx\dv\ds\dt \\
    &= \int_0^T \iint_{\R^{2n}} \int_{s}^T e^{s-t} \p_t (\phi(t,x+v(t-s),v)) M[f](s,x,v)\dt\dx\dv\ds \\
    &\quad - \int_0^T \int_0^t \iint_{\R^{2n}} e^{s-t} v\cdot \nabla_x \phi(t,x,v) M[f](s,x-v(t-s),v) \dx\dv\ds\dt.
\end{align*}
Integrating by parts with respect to $t$ on the first integral of the right-hand side, we obtain:
\begin{align*}
    &\int_0^T \iint_{\R^{2n}} \int_{s}^T e^{s-t} \p_t (\phi(t,x+v(t-s),v)) M[f](s,x,v)\dt\dx\dv\ds \\
    &= \int_0^T \iint_{\R^{2n}} \int_s^T e^{s-t} \phi(t,x+v(t-s),v) M[f](s,x,v)\dt\dx\dv\ds\\
    &\quad - \int_0^T \iint_{\R^{2n}} \phi(s,x,v) M[f](s,x,v) \dx\dv\ds \\
    &= \int_0^T \int_0^t \iint_{\R^{2n}} e^{s-t} \phi(t,x,v) M[f](s,x-v(t-s),v) \dx\dv\ds\dt \\
    &\quad - \int_0^T \iint_{\R^{2n}} \phi(s,x,v) M[f](s,x,v) \dx\dv\ds .
\end{align*}
Therefore,
\begin{align*}
    \calJ_2 &= \int_0^T \int_0^t \iint_{\R^{2n}} e^{s-t}\phi(t,x,v) M[f](s,x-v(t-s),v)\dx\dv\ds\dt \\
    &\quad - \int_0^T \iint_{\R^{2n}}\phi(t,x,v) M[f](t,x,v)\dx\dv\dt \\
    &\quad - \int_0^T \int_0^t \iint_{\R^{2n}} e^{s-t} v\cdot \nabla_x\phi(t,x,v) M[f](s,x-v(t-s),v) \dx\dv\ds\dt.
\end{align*}
Collecting all terms:
{\small\begin{align*}
    &\calJ_1 + \calJ_2 \\
    &= \int_0^T \iint_{\R^{2n}} \phi(t,x,v)\left(e^{-t}f_0(x-vt,v) + \int_0^t e^{s-t} M[f](s,x-v(t-s),v)\ds \right) \dx\dv\dt \\
    &\quad - \iint_{\R^{2n}} f_0(x,v) \,\phi_0(x,v)\,\dx\dv \\
    &\quad -\int_0^T \iint_{\R^{2n}} (v\cdot \nabla_x\phi) \left(e^{-t} f_0(x-vt,v) + \int_0^t e^{s-t} M[f](s,x-v(t-s),v) \ds\right)\dx\dv\dt \\
    &\quad - \int_0^T \iint_{\R^{2n}} \phi(t,x,v) \, M[f](t,x,v)\,\dx\dv\dt \\
    &= \int_0^T \iint_{\R^{2n}} \phi \, (f - M[f]) \,\dx\dv\dt - \iint_{\R^{2n}} f_0 \,\phi_0\,\dx\dv - \int_0^T \iint_{\R^{2n}} (v\cdot \nabla_x \phi) \,f \,\dx\dv\dt .
\end{align*}}
Recalling that $\calJ_1 + \calJ_2 = \int_0^T \iint_{\R^{2n}} f\,\p_t\phi\,\dx\dv\dt$, this completes the proof.
\end{proof}

\section*{Acknowledgments} Both authors deeply appreciate the helpful comments from Young-Pil Choi regarding this paper. This work is supported by NRF grant no. 2022R1A2C1002820 and RS-2024-00406127.

\bibliographystyle{alpha}

\begin{thebibliography}{99}

\bibitem{andersonwitting1974}{\sc J.L. Anderson and H.R. Witting}, \emph{A relativistic relaxation-time model for the Boltzmann equation}, Physica 74 (1974), pp.~466--488.

\bibitem{andriesaokiperthame2002}{\sc P. Andries, K. Aoki, and B. Perthame}, \emph{A consistent BGK-type model for gas mixtures}, J. Stat. Phys {106} (2002), pp.~993--1018.


\bibitem{baeklingenbergpirneryun2021}{\sc G.-C. Bae, C. Klingenberg, M. Pirner, and S.-B. Yun}, \emph{BGK model of the multi-species Uehling-Uhlenbeck equation}, Kinet. Relat. Models, {14} (2021), pp.~25--44.


\bibitem{berthelinbouchut2000}{\sc F. Berthelin and F. Bouchut}, \emph{Solution with finite energy to a BGK system relaxing to isentropic gas dynamics},
Ann. Fac. Sci. Toulouse Math., {9} (2000), pp.~605--630.

\bibitem{berthelinbouchut2002}{\sc F. Berthelin and F. Bouchut}, \emph{Relaxation to isentropic gas dynamics for a BGK system with single kinetic entropy}, Methods Appl. Anal., {9} (2002), pp.~313--327.

\bibitem{berthelinvasseur2005}{\sc F. Berthelin and A. Vasseur},
\emph{From Kinetic Equations to Multidimensional Isentropic Gas Dynamics Before Shocks},
SIAM J. Math. Anal. {36} (2005), pp.~1807--1835.

\bibitem{bhatnagargrosskrook1954}{\sc P.L. Bhatnagar, E.P. Gross, and M.L. Krook}, \emph{A model for collision processes in gases. I. Small amplitude processes in charged and neutral one-component systems}, Phys. Rev., {94} (1954), pp.~511--525. 

\bibitem{bobylevbisigroppispigapotapenko2018}{\sc A. V. Bobylev, M. Bisi, M. Groppi, G. Spiga, and I. F. Potapenko}, \emph{A general consistent BGK model for gas mixtures},
Kinet. Relat. Models, {11} (2018), pp.~1377--1393.


\bibitem{boltzmann1995}{\sc L. Boltzmann}, \emph{Lectures on gas theory}, University of California Press, Berkeley, 1964. Translated by Stephen G. Brush. Reprint of the 1896–1898 Edition. Reprinted by Dover Publications, 1995.


\bibitem{bouchut1999}{\sc F. Bouchut},
\emph{Construction of BGK Models with a Family of Kinetic Entropies for a Given System of Conservation Laws},
J. Stat. Phys. {95} (1999), pp.~113--170.

\bibitem{brullpavanschneider2012}{\sc S. Brull, V. Pavan, and J. Schneider}, \emph{Derivation of a BGK model for mixtures}, Europ. J. Mech. B/Fluids, {33} (2012),
pp.~74--86.

\bibitem{choihwang2024}{\sc Y.-P. Choi and B.-H. Hwang}, \emph{Global existence of weak solutions to a BGK model relaxing to the barotropic Euler equations}, 
Nonlinear Anal., {238} (2024), 113414.

\bibitem{dipernalionsmeyer1991}
{\sc R.J. Diperna, P.L. Lions, and Y. Meyer}, \emph{$L^p$ regularity of velocity averages}, Ann. Inst. H. Poincaré C Anal. Non Linéaire, {8.3--8.4} (1991), pp.~271--287.

\bibitem{hwang2024}
{\sc B.-H. Hwang}, \emph{Classical solutions to a BGK-type model relaxing to the isentropic gas dynamics}, preprint, arXiv:2402.09653.

\bibitem{karpermellettrivisa2013}
{\sc T. Karper, A. Mellet, and K. Trivisa}, \emph{Existence of Weak Solutions to Kinetic Flocking Models}, SIAM J. Math. Anal., {45} (2013), pp.~215--243. 

\bibitem{kimleeyun2021}{\sc D. Kim, M.-S. Lee, and S.-B. Yun}, \emph{Stationary BGK Models for Chemically Reacting Gas in a Slab}, J. Stat. Phys., 184.2 (2021), pp.~1--33.

\bibitem{klingenbergpirnerpuppo2017}{\sc C. Klingenberg, M. Pirner, and G. Puppo}, \emph{A consistent kinetic model for a two-component mixture with an application
to plasma}, Kinet. Relat. Models, {10} (2017), pp.~445--465.

\bibitem{pennisiruggeri2018}
{\sc S. Pennisi and T. Ruggeri}, \emph{A New BGK Model for Relativistic Kinetic Theory of Monatomic and Polyatomic Gases}, Journal of Physics: Conference Series 1035, (2018). 

\bibitem{perthamepulvirenti1993}{\sc B. Perthame and M. Pulvirenti}, \emph{Weighted $L^\infty$ bounds and uniqueness for the Boltzmann BGK model}, Arch. Ration. Mech. Anal., {125} (1993), pp.~289--295.

\bibitem{walender1954}{\sc P. Walender}, \emph{On the temperature jump in a rarefied gas}, Ark. Fys., {7} (1954), pp.~507--553.

    
\end{thebibliography}

\end{document}